\documentclass[11pt,a4paper,leqno,noamsfonts]{amsart}
\usepackage[english]{babel}
\usepackage[dvipsnames]{xcolor}
\usepackage{graphicx,pifont}
\usepackage[utopia]{mathdesign}
\usepackage[a4paper,top=3cm,bottom=3cm,
           left=3cm,right=3cm,marginparwidth=60pt]{geometry}
\usepackage[utf8]{inputenc}
\usepackage{braket,caption,comment,mathtools,stmaryrd,enumitem}
\usepackage[usestackEOL]{stackengine}
\usepackage{multirow,booktabs,microtype,relsize}
\usepackage[foot]{amsaddr}
\usepackage[colorlinks,bookmarks]{hyperref} %
      \hypersetup{colorlinks,%
            citecolor=britishracinggreen,%
            filecolor=black,%
            linkcolor=cobalt,%
            urlcolor=cornellred}
      \numberwithin{equation}{section}
\usepackage[capitalise]{cleveref}

\DeclareSymbolFont{usualmathcal}{OMS}{cmsy}{m}{n}
\DeclareSymbolFontAlphabet{\mathcal}{usualmathcal}
\DeclareMathAlphabet\BCal{OMS}{cmsy}{b}{n}

\makeatletter
\newcommand{\mylabel}[2]{#2\def\@currentlabel{#2}\label{#1}}
\makeatother

\newcommand\blfootnote[1]{%
  \begingroup
  \renewcommand\thefootnote{}\footnote{#1}%
  \addtocounter{footnote}{-1}%
  \endgroup
}
\definecolor{cornellred}{rgb}{0.7, 0.11, 0.11}
\definecolor{britishracinggreen}{rgb}{0.0, 0.26, 0.15}
\definecolor{cobalt}{rgb}{0.0, 0.28, 0.67}

\usepackage{amsmath,float,soul}
\DeclareMathOperator{\vd}{vdim}
\DeclareMathOperator{\rk}{rk}
\DeclareMathOperator{\vir}{\mathrm{vir}}

\newcommand{\dd}{\mathrm{d}}
\newcommand{\simto}{\,\widetilde{\to}\,}
\newcommand{\into}{\hookrightarrow}
\newcommand{\HH}{\mathrm{H}}
\newcommand{\OO}{\mathscr O}
\newcommand{\bfk}{\mathbf{k}}
\newcommand{\derived}{\mathbf{D}}
\newcommand{\BRPP}{\mathsf{RPP}}
\newcommand{\boldit}[1]{\boldsymbol{#1}}

\DeclareMathOperator{\bnu}{{\boldit{\nu}}} 
\DeclareMathOperator{\bn}{{\boldit{n}}}
\DeclareMathOperator{\sta}{\mathsf{st}}
\DeclareMathOperator{\free}{\mathsf{free}}
\DeclareMathOperator{\Ind}{\mathrm{Ind}}
\DeclareMathOperator{\Fact}{\mathrm{Fact}}
\DeclareMathOperator{\Subsoc}{\mathrm{Subsoc}}
\DeclareMathOperator{\Soc}{\mathrm{Soc}}

\DeclareMathOperator{\Hilb}{Hilb}

\DeclareMathOperator{\Var}{Var}
\DeclareMathOperator{\Spec}{Spec}
\DeclareMathOperator{\Sets}{Sets}
\DeclareMathOperator{\length}{length}
\DeclareMathOperator{\opp}{op}
\DeclareMathOperator{\Sch}{Sch}
\DeclareMathOperator{\Supp}{Supp}

\newcommand{\BA}{{\mathbb{A}}}
\newcommand{\BC}{{\mathbb{C}}}
\newcommand{\BE}{{\mathbb{E}}}
\newcommand{\BL}{{\mathbb{L}}}
\newcommand{\BN}{{\mathbb{N}}}
\newcommand{\BQ}{{\mathbb{Q}}}
\newcommand{\BY}{{\mathbb{Y}}}
\newcommand{\BZ}{{\mathbb{Z}}}
\newcommand{\CB}{{\mathcal B}}
\newcommand{\CE}{{\mathcal E}}
\newcommand{\CZ}{{\mathcal Z}}
\newcommand{\Fm}{{\mathfrak{m}}}

\newcommand{\FS}{{\mathfrak{S}}}

\usepackage[all]{xy}
\usepackage{tikz}
\usepackage{tikz-cd}
\usepackage{adjustbox}
\usepackage{rotating}
\usepackage{youngtab} 
\usepackage{ytableau} 

\theoremstyle{definition}

\newtheorem*{conventions*}{Conventions}

\newtheorem{definition}{Definition}[section]
\newtheorem{example}[definition]{Example}
\newtheorem{convention}[definition]{Convention}
\newtheorem{construction}[definition]{Construction}

\newtheoremstyle{thm} 
        {3mm}
        {3mm}
        {\slshape}
        {0mm}
        {\bfseries}
        {.}
        {1mm}
        {}
\theoremstyle{thm}
\newtheorem{theorem}[definition]{Theorem}
\newtheorem{corollary}[definition]{Corollary}
\newtheorem{lemma}[definition]{Lemma}
\newtheorem{prop}[definition]{Proposition}

\newtheorem{thm}{Theorem}

\theoremstyle{remark}
\newtheorem{remark}[definition]{Remark}
\newtheorem{notation}[definition]{Notation}

\newtheorem*{Acknowledgments*}{Acknowledgments}

\usetikzlibrary{patterns}

\title[The geometry of double nested Hilbert schemes of points on curves]{The geometry of double nested Hilbert schemes \\ of points on curves}
\author{Michele Graffeo, Paolo Lella, Sergej Monavari, Andrea T. Ricolfi, Alessio Sammartano}
\keywords{Hilbert schemes, 0-cycles, partitions, Grothendieck ring of varieties}
\subjclass[2020]{Primary 14C05; Secondary 14N35, 20M14, 05E14, 06A07.}
\thanks{S.M.~was  supported by the Chair of Arithmetic Geometry, EPFL. M.G., P.L. and A.S.~were partially supported by the project PRIN 2020 \lq\lq Squarefree Gr\"obner degenerations, special varieties and related topics\rq\rq~(MUR, project number 2020355B8Y). A.R.~was partially supported by the project PRIN 2022 ``Geometry of algebraic structures: moduli, invariants, deformations'' (MUR, project number 2022BTA242). All authors are members of the INdAM group GNSAGA}

\begin{document}
\begin{abstract}
Let $C$ be a smooth curve. In this paper we investigate the geometric properties of the \emph{double nested Hilbert scheme of points} on $C$, a moduli space introduced by the third author in the context of BPS invariants of local curves and sheaf counting on Calabi--Yau 3-folds. We prove this moduli space is connected, reduced and of pure dimension; we list its components via an explicit combinatorial characterisation and we show they can be resolved, when singular, by products of symmetric products of $C$. We achieve this via a purely algebraic analysis of the factorisation properties of the monoid of \emph{reverse plane partitions}. 
We discuss the (virtual) fundamental class of the moduli space, we describe the local equations cutting it inside a smooth ambient space, and finally we provide a closed formula for its motivic class in the Grothendieck ring of varieties. 
\end{abstract}

\maketitle

{\hypersetup{linkcolor=black}\tableofcontents}

\section{Introduction}

\subsection{Overview}
The main character of this paper is a moduli space called the \emph{double nested Hilbert scheme of points}. It was introduced by the third author in \cite{Mon_double_nested}, where its link with enumerative geometry, e.g.~with the theory of stable pairs and BPS states on Calabi--Yau 3-folds, is discovered and explored in detail (cf.~\Cref{subsec:EG} for more on this). 

The input data is a variety $X$, a partition (or Young diagram) $\lambda \subset \BZ^2$ and a \emph{reverse plane partition} 
\[
\bn = (\bn_{\Box})_{\Box \in \lambda}, \quad \bn_{\Box} \in \BZ_{\geq 0}
\]
of shape $\lambda$. The latter is a way to attach to every box $\Box \in \lambda$ a natural number $\bn_{\Box}$, in such a way that each column and each row of the Young diagram $\lambda$ carries a nondecreasing sequence of natural numbers. Out of this data, one constructs a scheme $X^{[\bn]}$ whose closed points correspond to configurations of $0$-dimensional closed subschemes of $X$
\begin{equation*}
\begin{tikzcd}
Z_{0,0}\arrow[r,hook]\arrow[d,hook] &Z_{1,0}\arrow[d,hook]\arrow[r,hook]&Z_{2,0}\arrow[d,hook]\arrow[r,hook]\arrow[d,hook]&Z_{3,0}\arrow[r,hook]\arrow[d,hook]&\cdots\\      Z_{0,1}\arrow[r,hook]\arrow[d,hook]&Z_{1,1}\arrow[d,hook]\arrow[r,hook]&Z_{2,1}\arrow[d,hook]\arrow[r,hook]&Z_{3,1}\arrow[d,hook]\arrow[r,hook]&\cdots\\
Z_{0,2}\arrow[d,hook]\arrow[r,hook]&Z_{1,2}\arrow[d,hook]\arrow[r,hook]&Z_{2,2}\arrow[d,hook]\arrow[r,hook]& \rotatebox[origin=c]{-45}{$\cdots$} &\\
          \vdots&\vdots&\vdots& &\\
\end{tikzcd}  
\end{equation*}
where the inclusions follow the shape of $\lambda$ and $\chi(\OO_{Z_{\Box}}) = \bn_{\Box}$ for each $\Box \in \lambda$.

This moduli space is a generalisation of the classical \emph{nested Hilbert scheme of points}, which in turn generalises the even more familiar \emph{Hilbert scheme of points} $X^{[n]}=\Hilb^nX$ (cf.~\Cref{subsec:open-problems-Hilb}), parametrising $0$-dimensional closed subschemes $Z \into X$ with $\chi(\OO_Z)=n$. 

The purpose of this paper is to study the geometry of $X^{[\bn]}$ in the case where $X=C$ is a smooth, quasiprojective irreducible curve over an algebraically closed field $\bfk$ of characteristic $0$.
We can sum up our results, informally, as follows:
\begin{itemize}
    \item [\mylabel{moduli-1}{(1)}] The intersection of  all irreducible components of $C^{[\bn]}$ contains a copy of $C$; in particular, $C^{[\bn]}$ is connected (cf.~\Cref{cor:conn}). Moreover, $C^{[\bn]}$ is reduced (cf.~\Cref{thm: complete interse}).
    \item [\mylabel{moduli-2}{(2)}] The scheme $C^{[\bn]}$ is pure, with number of irreducible components only depending on $\bn$ (cf.~\Cref{thm:irr components}). Moreover, we have an explicit formula for its dimension (cf.~\Cref{prop:weight-formula}).
    \item [\mylabel{moduli-3}{(3)}] The irreducible components $V \subset C^{[\bn]}$ are in general not normal, and can be resolved by products of symmetric products of $C$, which coincide with their normalisations (cf.~\Cref{prop:canonical cover}).  We provide a combinatorial classification of the components $V \subset C^{[\bn]}$ (cf.~\Cref{thm:irr components}), a criterion for when a component is singular (cf.~\Cref{thm: main thm for criterion of smooth}), and we exhibit a component $C^{[\bn]}_{\sta}\subset C^{[\bn]}$ which is \emph{always} nonsingular (cf.~\Cref{prop:standardsmooth}).
    \item [\mylabel{moduli-4}{(4)}] The scheme $C^{[\bn]}$ is the zero locus of a section of a vector bundle on a smooth variety $A_{C,\bn}$ and, moreover, the immersion $C^{[\bn]} \into A_{C,\bn}$ is \emph{regular} (cf.~\Cref{thm: complete interse}). 
    \item [\mylabel{moduli-5}{(5)}] The scheme $C^{[\bn]}$ admits explicit \'etale-local equations cutting it out inside \emph{two} different  affine spaces (cf.~\Cref{sec:local-eqns}). The code for computing those equations is available for \textit{Macaulay2} in the package \href{www.paololella.it/software/DoubleNestedHilbertSchemes.m2}{\tt DoubleNestedHilbertSchemes.m2} (cf.~\cref{app:computations}).
    \item [\mylabel{moduli-6}{(6)}] The fundamental class of $C^{[\bn]}$, a homogeneous cycle class by \ref{moduli-2}, agrees with the virtual fundamental class $[C^{[\bn]}]^{\vir} \in A_\ast (C^{[\bn]})$ induced by \ref{moduli-4}, which implies that all virtual intersection numbers on the moduli space are computable by standard (not virtual) methods (cf.~\Cref{cor:VFC=FC}).
    \item [\mylabel{moduli-7}{(7)}] The motivic generating function 
    \[
    \sum_{\bn\in\BRPP^\lambda}\,\bigl[C^{[\bn]}\bigr] \boldit{q}^{\bn}\,\in\, K_0(\Var_{\bfk})\llbracket q_{\Box}\,|\,\Box \in \lambda \rrbracket
    \]
     is computed via an explicit formula in terms of the zeta function of $C$ (cf.~\Cref{cor:motive-curve-dn}), and is a rational function. In particular, all motivic-in-nature invariants of the moduli space are computable via generalised Euler characteristics, i.e.~ring homomorphisms out of $K_0(\Var_{\bfk})$.
\end{itemize}

In the next subsection we give further details about \ref{moduli-2}, \ref{moduli-3} and \ref{moduli-7}.

\subsection{Main results}
Let $\lambda$ be a partition, endowed with its natural poset structure. A \emph{reverse plane partition} of shape $\lambda$ is a poset morphism $\bn \colon \lambda \to \BN$. Reverse plane partitions of shape $\lambda$ form a cancellative commutative monoid with respect to addition, denoted $\BRPP^\lambda$. In \Cref{PropositionAtoms}, we show that every element of the monoid $\BRPP^\lambda$ can be expressed as linear combination of irreducible elements. The \emph{derivative} of a reverse plane partition $\bn\in\BRPP^\lambda$ (cf.~\Cref{def:derivative}) is the map $\Delta \bn \colon \lambda \to \BZ$ given by 
\[
\Delta \bn({i,j})=
        \bn({i,j})-  \bn({i-1,j})- \bn({i,j-1})+ \bn({i-1,j-1}).
\]
Here we identify boxes $\Box \in \lambda$ with their coordinates $(-,-)\in\BZ^2$ (cf. \Cref{convention-boxes}).
The \emph{size} of $\Delta \bn$, by definition, is the quantity
\[
\omega(\boldit{n}) = \sum_{(i,j) \in \lambda} \Delta \bn({i,j}). 
\]
It can be easily expressed in terms of the values  $\bn_{\Box}$, for those boxes $\Box$ belonging to the \emph{socle} and to the \emph{subsocle} of $\lambda$ (cf.~\Cref{prop:weight-formula}).

There are special reverse plane partitions called \emph{Young indicators} (cf.~\Cref{def:Young-Indicator}). Their set is denoted $\Ind(\lambda)$. A \emph{factorisation} of $\boldit{n} \in \BRPP^\lambda$ (cf.~\Cref{def:factorisation}) is a tuple $(n_{\bnu})_{\bnu \in \Ind(\lambda)} \in \BN^{\Ind(\lambda)}$ such that 
\[
\boldit{n} = \sum_{\bnu \in \Ind(\lambda)} n_{\bnu}\cdot\bnu.
\]
The sum $\sum_{\bnu\in \Ind(\lambda)}n_{\bnu} \in \BN$ of the coefficients appearing in the factorisation is called the \emph{length} of the factorisation.

The following is our first main result.

\begin{thm}[\Cref{TheoremRPPHalfFactorial}, \Cref{thm:irr components}]\label{thm:main-A}
Let $C$ be a smooth curve. Let $\lambda$ be a partition, and fix $\bn \in \BRPP^\lambda$. Then,
\begin{itemize}
    \item [\mylabel{main1-1}{\normalfont{(1)}}]  irreducible components of $C^{[\boldit{n}]}$ correspond bijectively to factorisations of $\boldit{n}$, and 
    \item [\mylabel{main1-2}{\normalfont{(2)}}] all factorisations of $\boldit{n}$ have length $\omega(\boldit{n})$, which equals the dimension of the corresponding irreducible component of $C^{[\boldit{n}]}$. 
\end{itemize}
In particular, $C^{[\boldit{n}]}$ is pure of dimension $\omega(\boldit{n})$.
\end{thm}

A description of the irreducible components of $C^{[\boldit{n}]}$ has independently been obtained by Schimpf in \cite{Schimpf}, where enumerative applications to stable pair theory are also explored.

Since the moduli space $C^{[\bn]}$ may be reducible, it is in general singular. Its irreducible components may happen to be singular as well. In order to describe a resolution of each component $V\subset C^{[\bn]}$, in \Cref{prop:canonical cover} we construct a finite birational morphism 
\[
\begin{tikzcd}
\displaystyle\prod_{\bnu \in \Ind(\lambda)} C^{(n_{\bnu})} \arrow{r}{\varphi_V} & V,
\end{tikzcd}
\]
where we denote by $C^{(m)}$ the $m$-th symmetric product $C^m / \FS_m$ of $C$. We prove that $\varphi_V$ coincides with the normalisation of $V$, where $(n_{\bnu})_{\bnu \in \Ind(\lambda)}$ is the   factorisation corresponding to $V$ under the bijection of        \Cref{thm:main-A}\,\ref{main1-1}.

We prove the following result, describing criteria for nonsingularity of the components.

\begin{thm}[\Cref{thm:smoothness-components}, \Cref{thm: main thm for criterion of smooth}]\label{mainB}
Let $C$ be a smooth curve. Let $\lambda$ be a partition, and fix $\bn \in \BRPP^\lambda$. Let $V\subset C^{[\bn]}$ be an irreducible component, with associated factorisation $T=(n_{\bnu})_{\bnu \in \Ind(\lambda)}$. The following conditions are equivalent:
\begin{itemize}
    \item [{\normalfont{(a)}}] $V$ is smooth,
    \item [{\normalfont{(b)}}] $V$ is normal,
    \item [{\normalfont{(c)}}] $\varphi_V$ is an isomorphism,
    \item [{\normalfont{(d)}}] $\varphi_V$ has injective differential,
    \item [{\normalfont{(e)}}] there is   no nontrivial relation of the form
\begin{equation*}\label{eq:singrelation2}
        \sum_{\bnu\in \Ind(\lambda)}m_{\bnu}\cdot\bnu=0,
    \end{equation*}
where $m_{\bnu}\in \BZ$ are integers and are possibly nonzero only for the indicators $\bnu$ appearing in the factorisation $T$.
\end{itemize}
\end{thm}
Some purely combinatorial criteria ensuring smoothness of certain classes of distinguished irreducible components are given in Propositions \ref{prop:standardsmooth}, \ref{prop:completesmooth}.

\smallbreak

The last result we want to mention in this introduction is the formula expressing the motive 
\begin{equation}\label{eqn:motive-Cn}
\bigl[C^{[\boldit{n}]}\bigr] \,\in\, K_0(\Var_{\bfk}).
\end{equation}
In \Cref{sec:multiple-ps}, we recall the notion of (multiple variable) \emph{power structure} on the Grothendieck ring of varieties $K_0(\Var_{\bfk})$. The following result expresses the class \eqref{eqn:motive-Cn} as a product of motivic zeta functions
\[
\zeta_C(t) = 1 + \sum_{n > 0}\,\bigl[C^{(n)}\bigr] t^n \,\in\, K_0(\Var_{\bfk})\llbracket t \rrbracket
\]
in appropriate variables. This series was introduced by Kapranov in \cite{Kapranov_rational_zeta}, where its rationality is proved.

\begin{thm}[\Cref{cor:motive-curve-dn}]\label{thm:main-C}
Let $C$ be a smooth curve, $\lambda$ a Young diagram. Introduce a formal variable $q_{\Box}$ for every $\Box \in \lambda$, and set $p_{\Box} = \prod_{\Box' \in H(\Box)} q_{\Box'}$, where $H(\Box)$ is the hook of $\Box$ (cf.~\Cref{notation:hooks}). There is an identity
\[
\sum_{\boldit{n} \in \BRPP^\lambda} \bigl[C^{[\boldit{n}]}\bigr] \boldit{q}^{\boldit{n}}= \prod_{\Box \in \lambda} \, \zeta_C(p_{\Box}) \, \in\, K_0(\Var_{\bfk}) \llbracket q_{\Box} \,|\,\Box \in \lambda \rrbracket.
\]
In particular, the left hand side is a rational function in the variables $q_{\Box}$.
\end{thm}
\subsection{Some open problems in the subject}\label{subsec:open-problems-Hilb}
Let $X$ be a smooth irreducible variety of dimension $d$. The Hilbert scheme of points $X^{[n]}$ agrees with the symmetric product $X^{(n)} = X^n / \mathfrak S_n$ if $d=1$. In general, $X^{[n]}$ if smooth if and only $d\leq 2$ or $n \leq 3$.

It is known that $X^{[n]}$ becomes pathological if $d \geq 4$ \cite{Pathologies_Hilb}. Very few \emph{geometric} properties are known when $d = 3$, which remains to date the most mysterious case. For instance, we know that if $X$ is a smooth 3-fold, $X^{[n]}$ is irreducible for $n \leq 11$ (see \cite{Douvropoulos2017} and the references therein) and reducible for $n \geq 78$ \cite{IARRO}. However, we do \emph{not} know:
\begin{itemize}
    \item [\mylabel{oq-i}{(i)}] whether $X^{[n]}$ is generically reduced, or whether it satisfies Vakil's Murphy's Law \cite{zbMATH05033665}, 
    \item [\mylabel{oq-ii}{(ii)}] how many components (at most) $X^{[n]}$ has,
    \item [\mylabel{oq-iii}{(iii)}] an example of a nonsmoothable subscheme $Z \into X$.
\end{itemize}
Recently it was proven in \cite{JKS-Behrend-nonconstant} that the Behrend function of $\Hilb^n(\BA^3)$ is nonconstant for $n \geq 24$. This is an indication that the Hilbert scheme of a smooth 3-fold may be generically nonreduced, see \cite{ricolfi-constant-behrend}.
We refer the reader to \cite{Hilb-open-problems} for a recent overview on open problems around Hilbert schemes.

\smallbreak
Let $\underline{n} = (n_1 \leq n_2 \leq \cdots \leq n_p)$ be a nondecreasing sequence of nonnegative integers.
The nested Hilbert scheme $X^{[\underline{n}]}$ on a smooth $d$-dimensional variety $X$ parametrises  chains of 0-dimensional closed subschemes 
\[
Z_1 \into Z_2 \into \cdots \into Z_p \into X
\]
with prescribed lengths
$\chi(\OO_{Z_{i}}) = n_i$.
When $d = 1$, the nested Hilbert scheme $X^{[\underline{n}]}$ is isomorphic to a product of symmetric products of $X$.
Much of the research focuses on the case when $d=2$, which yields a singular moduli space unless $p=2$ and $n_2-n_1 =1$ \cite{Che_cellular_decomposition}.
There is only partial information about the singularities and irreducible components of $X^{[\underline{n}]}$.
For example, $X^{[\underline{n}]}$ is always irreducible if $p=2$ and it can be reducible if $p \geq 5$,
but we do not know what happens when $p =3,4$.
We know very little about the types of singularities that can occur, and we do not   know whether it is always reduced, not even when $p=2$. It is known that $X^{[2,n]}$ has only rational singularities if $\dim X=2$. Moreover, in dimension $d=4$ the nested Hilbert scheme $X^{[\underline n]}$ has a generically nonreduced component already for $\underline n = (1\leq 8)$.
We refer the reader to 
\cite{gangopadhyay2022irreducibility,giovenzana2024unexpected,RamkumarSammartano,RyanTaylor}
for more details and recent progress on this topic.

\subsection{Relation to enumerative geometry}\label{subsec:EG}
The double nested Hilbert scheme was introduced by the third author \cite{Mon_double_nested} in relation to the study of \emph{stable pair invariants} of a \emph{local curve}, that is a quasiprojective 3-fold $X$ obtained as the total space of the direct sum of two line bundles $L_1, L_2$ on a smooth projective curve $C$. In this setting, the relevant virtual invariants of $X$ are computed from the moduli space of stable pairs \cite{PT_curve_counting_derived}, but can be expressed in terms of virtual intersection numbers on $C^{[\boldit{n}]} $ by Graber--Pandharipande virtual localisation \cite{GP_virtual_localization}. This is achieved by noticing that the torus action scaling the fibres of $X$ lifts to the moduli space of stable pairs, and that the connected components of the  fixed locus correspond to double nested Hilbert schemes $C^{[\boldit{n}]}$, cf. \cite{Mon_double_nested}. Thanks to Corollary \ref{cor:VFC=FC},  our work shows, in particular, that all virtual intersection numbers computed in \cite{Mon_double_nested} coincide with nonvirtual intersection numbers of $C^{[\boldit{n}]}$, which can be furthermore expressed as intersection numbers on symmetric products (and Jacobians)  of curves, where  standard techniques of intersection theory apply \cite{ACGH_Volume_1}, cf. Section \ref{sec: virtual fundamental}. 

The study of enumerative geometry in the case of local curves is a key step towards the study of the more general case of  projective Calabi--Yau 3-folds. See for instance the recent work of Pardon \cite{pardon2023universally}, where the author deduces the correspondence between  Gromov--Witten and stable pair invariants for projective Calabi--Yau 3-folds --- first conjectured in \cite{MNOP_1} --- from the analogous statement for local curves, studied e.g. in \cite{BP_local_GW_curves, Mon_double_nested, OP_local_theory_curves}. 

The understanding of the geometry of the double nested Hilbert schemes --- and its motivic invariants --- could play a role in the computation of \emph{refined} BPS invariants  predicted in \cite{CDDP_parabolic} via string-theoretic dualities. Indeed, following the circle of ideas of Nekrasov--Okounkov from the point of view of M-theory \cite{NO_membranes_and_sheaves}, it should be possible to  relate such refined BPS invariants to equivariant $K$-theoretic stable pair invariants of local curves (see e.g.~\cite{CKM_K_theoretic, KOO_2_legDT, Mon_double_nested, Okounk_Lectures_K_theory}).

\begin{conventions*}
We work over an algebraically closed field $\bfk$ of characteristic $0$. A `variety' is a quasiprojective integral scheme of finite type over $\bfk$. For a variety $X$, we denote by $X^{(n)}$ its $n$-th symmetric product $X^n / \mathfrak S_n$. By the word `curve' we mean a $1$-dimensional variety. In particular, a smooth variety will be connected. The irreducible components of a scheme are always considered with the reduced induced closed subscheme structure. Given two sets $A,B$, we denote by $A^B$ the set of all functions $B \to A$. When $A = \BZ$ (resp.~$A = \BN$) and $B$ is finite, then $A^B$ is the set of tuples of integers (resp.~natural numbers) indexed by $B$.
\end{conventions*}

\subsection*{Acknowledgements}
S.M.~is grateful to Francesca Carocci and Maximilian Schimpf for useful conversations, and to Rahul Pandharipande for encouraging us to work this problem out. We are grateful to Joachim Jelisiejew for helpful discussions and for spotting a mistake in the first draft of this paper. We also thank the anonymous referee for several helpful comments, which improved the exposition of the paper.

\section{The monoid of reverse plane partitions}\label{subsec:combinatorics}
In this section, we introduce and study the monoid of reverse plane partitions of fixed shape.
As we will see in the next sections, its combinatorics governs the geometry of the double nested Hilbert scheme of points on smooth curves, the main character of this paper. This monoid is also interesting in its own right, from the point of view of combinatorics, factorisation theory, commutative algebra, and toric geometry.

\subsection{Young diagrams and reverse plane partitions}\label{subsec:YD-and-RPP}
A \emph{partition} $\lambda$ is a finite nondecreasing sequence of positive integers 
$\lambda=(\lambda_0, \lambda_1, \ldots,\lambda_{l(\lambda)-1})$. The number of parts $l(\lambda)$ is called the \emph{length} of $\lambda$, and the number
\[
\lvert \lambda \rvert = \sum_{i=0}^{l(\lambda)-1} \lambda_i
\]
is called the \emph{size} of $\lambda$.
A partition $\lambda$ can be equivalently described by its associated \emph{Young diagram},
which is the collection of  pairs  in $(i,j) \in \BN^2$  such that $0\leq j< \lambda_{i}$. In particular, $\lambda$ can be seen as a subset of $\BN^2$.
We emphasise that pairs start at $(0,0)$ in our notation. 

Young diagrams are represented pictorially as collections of boxes stacked in the corner of a `two-dimensional room'. For example, the Young diagram 
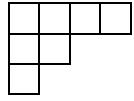
\[
\begin{matrix}
       \yng(4,2,1) \quad 
       \ytableausetup{boxsize=1.09em}
       \ytableausetup{boxframe=0.02em}\ytableausetup{aligntableaux=bottom}
   \end{matrix}
\]
corresponds to the partition $\lambda = (3,2,1,1)$, which satisfies  $l(\lambda)=4$ and $\lvert \lambda \rvert = 7$. The length of a partition is then nothing but the number of columns in the associated Young diagram, and the size is the number of boxes.

As is common in the literature, we  identify partitions and Young diagrams, and use the same letter to denote the two corresponding objects. 

\begin{convention}\label{convention-boxes}
We identify a box $\Box \in \lambda$ with its coordinates $(i,j) \in \BN^2$, and simply write $(i,j) \in \lambda$, with the convention that the first coordinate (resp.~the second coordinate) grows in the rightward (resp.~downward) direction.
\end{convention}

We endow Young diagrams with the poset structure inherited from the  componentwise partial order in $\BZ^2$,
that is,  
\begin{equation}\label{eqn:partial-order-lambda}
(i,j)\le (i',j') \qquad \Leftrightarrow \qquad i\le i' \text{~and~} j\le  j'.
\end{equation}

\begin{definition}\label{def:Young-Filling}
Let $\lambda$ be a Young diagram.
A \emph{Young filling} of shape $\lambda$ is a function  $\bn \colon \lambda \to \BZ$.
The value of the Young filling at a box $(i,j) \in \lambda$ is denoted by $\bn(i,j)$.
We say that a \emph{Young filling} is nonnegative if $\bn(i,j)\ge 0$ for all $(i,j) \in \lambda$.
\end{definition}

\begin{convention}\label{con:zero-outside}
Let $\bn$ be a Young filling. It will sometimes be useful to adopt the convention that $\bn(i,j)=0$ for all $(i,j)\in \BZ^2\smallsetminus \lambda$, i.e., to extend $\bn$ by zero to the whole of $\BZ^2$.
\end{convention}

We denote the sets of Young fillings and of nonnegative Young fillings of a given shape $\lambda $ by $\BZ^\lambda$ and $\BN^\lambda$, respectively.
The \emph{size} of a Young filling is the integer
\[
\lvert \bn \rvert = 
\sum_{(i,j) \in\lambda}\bn(i,j).
\]
We represent Young fillings pictorially by filling the underlying Young diagrams with integers,
for example, the following is a nonnegative Young filling of shape $(3,2,1,1)$ and size 14.
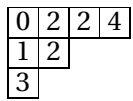
\[
\begin{matrix}
\begin{ytableau}
0 & 2 &  2 & 4 \\
1 & 2 \\
3
\end{ytableau}
\end{matrix}
\]

\Cref{def:RPP}  introduces the main character of this section.

\begin{definition}\label{def:RPP}
Let $\lambda$ be a Young diagram.
A \emph{reverse plane partition} of shape $\lambda$ is a nonnegative Young filling
$\bn\colon \lambda \to \BN$ that is nondecreasing with respect to the partial order \eqref{eqn:partial-order-lambda} of $\lambda$.
\end{definition}

In other words, a reverse plane partition is a poset morphism between the posets $\lambda$ and $\BN$.
In down to earth terms, it is a way to fill in a Young diagram with nonnegative integers that is nondecreasing along rows and columns.
Reverse plane partitions appear naturally in combinatorics and representation theory 
\cite{GPT,HG_reversed_plane_partitions,Stanley}.

Two Young fillings of the same shape can be added, yielding another Young filling of the same shape.
Then,  $\BZ^\lambda$  is a finitely generated free abelian group,
and $\BN^\lambda$ is a finitely generated free abelian monoid, with respect to addition. 
In this section, we are interested in the following submonoid of  $\BN^\lambda$.

\begin{notation}
Let $\lambda$ be a Young diagram.
The set of  reverse plane partitions of shape $\lambda$ is denoted by $\BRPP^\lambda$.
It is a cancellative commutative monoid with respect to addition.
\end{notation}

\begin{remark}
The monoid $\BRPP^\lambda$ is a special case of the more general class of monoids 
consisting of  nondecreasing functions from an arbitrary  poset to $\BN$. 
They are treated in \cite{BGSM}, in relation  to some applications to matroid theory. 

The specular version of this construction is the class of monoids of nonincreasing functions from a poset to $\BN$. 
They play a prominent role in combinatorics, appearing in the literature  under the names of $P$-partitions or weak $P$-partitions 
\cite{FR,Garsia,Stanley}. 
When the poset is a Young diagram, one obtains the familiar notion of \emph{plane partitions}.

The two constructions are equivalent, and the equivalence is obtained by  taking the opposite poset, where all the inequalities are reversed.
However, 
we point out that the opposite poset of a Young diagram is not a Young diagram, 
and,
consequently,
the monoid of reverse plane partitions of shape $\lambda$
is not  isomorphic to the monoid of plane partitions of shape $\lambda$. As an example,
consider the Young diagram  $\lambda = (2,1)$.
Then,   $\BRPP^\lambda$ has three irreducible elements (see \Cref{PropositionAtoms} below)
\[
\begin{matrix}
\begin{ytableau}
0 &1\\
0 \\
\end{ytableau}
\end{matrix}\,\qquad
\begin{matrix}
\begin{ytableau}
0 &0\\
1 \\
\end{ytableau}
\end{matrix}\,\qquad
\begin{matrix}
\begin{ytableau}
1 &1\\
1 \\
\end{ytableau}
\end{matrix}\,
\]
whereas the monoid of nonincreasing functions $\lambda \to \BN$ has the four Young fillings 
\[
\begin{matrix}
\begin{ytableau}
1 &0\\
0 \\
\end{ytableau}
\end{matrix}\,\qquad
\begin{matrix}
\begin{ytableau}
1 &1\\
0 \\
\end{ytableau}
\end{matrix}\,\qquad
\begin{matrix}
\begin{ytableau}
1 &0\\
1 \\
\end{ytableau}
\end{matrix}\,\qquad
\begin{matrix}
\begin{ytableau}
1 &1\\
1 \\
\end{ytableau}
\end{matrix}
\]
as irreducible elements.
\end{remark}

\subsection{Factorisations in \texorpdfstring{$\BRPP^\lambda$}{RPP}}\label{subsec:fact,der}
Next, we study factorisations (cf.~\Cref{def:factorisation}) in the monoid $\BRPP^\lambda$. To do so, we need to introduce some basic definitions.

\begin{definition}\label{upperset-connected}
Let $\lambda$ be a partition.
An \emph{upper set} is a  subset $U \subset \lambda $ such that if $(i,j) \in U$ and  $(i',j')\geq (i,j)$, then $(i',j') \in U$.
Equivalently, an upper set is the complement of another Young diagram in $\lambda$.

A subset $V \subset \lambda$ is \emph{connected} if any two boxes in $V$ are joined by a path of boxes of $V$
such that any two consecutive boxes share an edge. 
\end{definition}

Let $U \subset \lambda$ be an upper set. 
The \emph{indicator} of $U$ is the nonnegative Young filling $\chi_U \colon \lambda \to \BN$ defined by 
\[
\chi_U(i,j) = \begin{cases}
1 & \text{if}\quad (i,j) \in U,\\
0 & \text{if}\quad (i,j) \notin U.
\end{cases}
\]
Observe that $\chi_U$ is a reverse plane partition.

\begin{definition}\label{def:Young-Indicator}
We define a \emph{Young indicator} to be the indicator of a nonempty connected upper set. We denote by $\Ind(\lambda)$ the set of Young indicators of shape $\lambda$.
\end{definition}

A nonzero element $x$ of a monoid is  \emph{irreducible} if $x=y+z$ forces $y=0$ or $z=0$.

\begin{prop}\label{PropositionAtoms}
Let $\lambda$ be a Young diagram.
\begin{itemize}
\item [\mylabel{atoms-1}{\normalfont{(1)}}] Every reverse plane partition  is a linear combination of Young indicators with nonnegative integer coefficients. 
\item [\mylabel{atoms-2}{\normalfont{(2)}}] A reverse plane partition  is irreducible if and only if it is a Young indicator.
\item [\mylabel{atoms-3}{\normalfont{(3)}}] The set of Young indicators is the unique minimal set of generators of $\BRPP^\lambda$.
\end{itemize}
\end{prop}

\begin{proof}
We start proving \ref{atoms-1}. First, we show that every reverse plane partition $\bn\colon \lambda \to \BN$ is a linear combination of indicators.
Let $\{k_1 < \cdots < k_t\} = \bn(\lambda)\subset \BN$ be the image of the function $\bn$.
Then, the set  $V_h = \bn^{-1}(\{k_h,k_h+1,\ldots\})\subset \lambda$ is an upper set for each $h$,
and we have 
$$\bn = k_1 \chi_{V_1} + (k_2-k_1) \chi_{V_2} + \cdots + (k_t-k_{t-1})\chi_{V_t}.$$
We now show that every indicator $\chi_U$ is a sum of Young indicators.
Observe that connectedness is an equivalence relation.
Denote by $U_1, \ldots, U_r$ the (nonempty) connected parts of $U$. Then, each $U_i$ is again an upper set,
so each $\chi_{U_i}$ is a Young indicator.
We have
$\chi_U = \chi_{U_1} + \chi_{U_2} + \cdots + \chi_{U_r}$,
as desired.

The proof of \ref{atoms-2} follows just as in  \cite[Prop.~2]{BGSM}.
Finally, \ref{atoms-3} follows combining \ref{atoms-1} and \ref{atoms-2}.
\end{proof}

\Cref{PropositionAtoms} says that $\BRPP^\lambda$ is an \emph{atomic} monoid, that is, 
every element can be expressed as linear combination of irreducible elements.

\begin{definition}\label{def:factorisation}
Let $\bn \in \BRPP^\lambda$ be a reverse plane partition. 
A \emph{factorisation} of $\bn$ is a tuple of nonnegative integers
\[
\left(n_{\bnu}\right)_{\bnu \in \Ind(\lambda)} 
\in \BN^{\Ind(\lambda)}
\quad\text{such that}\quad
\bn = \sum_{\bnu \in \Ind(\lambda)} n_{\bnu} \cdot \bnu.
\]
The set of Young indicators $\Ind(\lambda)$ was introduced in \Cref{def:Young-Indicator}. The \emph{length} of the factorisation is the integer $\sum_{\bnu \in \Ind(\lambda)} n_{\bnu}$.
We denote by $\Fact(\bn)$ the set of all factorisations of a reverse plane partition $\bn$.  
\end{definition}

\begin{example}\label{Example3Factorisations}
The Young indicators of shape $\lambda = (2,2)$ are
\[
\begin{tikzpicture}
    \node at (0,0) {\young(01,11)};
    \node at (-0.85,-0.05) {$\bnu_1=$};
\end{tikzpicture}
\quad
\begin{tikzpicture}
    \node at (0,0) {\young(00,01)};
    \node at (-0.85,-0.05) {$\bnu_2=$};
\end{tikzpicture}
\quad
\begin{tikzpicture}
    \node at (0,0) {\young(01,01)};
    \node at (-0.85,-0.05) {$\bnu_3=$};
\end{tikzpicture}
\quad
\begin{tikzpicture}
    \node at (0,0) {\young(00,11)};
    \node at (-0.85,-0.05) {$\bnu_4=$};
\end{tikzpicture}
\quad
\begin{tikzpicture}
    \node at (0,0) {\young(11,11)};
    \node at (-0.85,-0.05) {$\bnu_5=$};
\end{tikzpicture}
\]
The reverse plane partition
\[
\begin{tikzpicture}
    \node at (1.3,0.05) {$\in\,\,\, \BRPP^\lambda$};
    \node at (0,0) {\young(02,24)};
    \node at (-0.85,-0.03) {$\bn=$};
\end{tikzpicture}
\]
has three factorisations, namely
\begin{center}
    \begin{tikzpicture}
    \node at (-1.25,0) {$\bn$};
    \node at (0,0) {\young(01,11)};
    \node at (-0.75,0) {$= 2\cdot$};
    \node at (1.55,0) {\young(00,01)};
    \node at (0.8,0) {$+ 2\cdot$};
    
    \node at (0,-1.2) {\young(01,01)};
    \node at (-0.75,-1.2) {$= 2\cdot$};
    \node at (1.55,-1.2) {\young(00,11)};
    \node at (0.8,-1.2) {$+ 2\cdot$};
    
    \node at (0,-2.4) {\young(01,11)};
    \node at (-0.93,-2.45) {$= $};
    \node at (1.55,-2.4) {\young(00,01)};
    \node at (0.77,-2.4) {$+$};
    \node at (3.1,-2.4) {\young(01,01)};
    \node at (2.32,-2.4) {$+$};
    \node at (3.87,-2.4) {$+$};
    \node at (4.65,-2.4) {\young(00,11)};
    \end{tikzpicture}
\end{center}
and one has
\[
\Fact(\bn) = \big\{ (2,2,0,0,0), (0,0,2,2,0), (1,1,1,1,0)\big\}.
\]
Notice that all factorisations of $\bn$ have length 4.
\end{example}

\begin{definition}\label{def:derivative}
The \emph{derivative} of a reverse plane partition $\bn\colon \lambda \to \BN$ is the Young filling $\Delta \bn \colon \lambda \to \BZ$ given by 
   \begin{equation}\label{eq:derivative}
       \Delta \bn({i,j})=
        \bn({i,j})-  \bn({i-1,j})- \bn({i,j-1})+ \bn({i-1,j-1}).
   \end{equation}
The \emph{weight} of a reverse plane partition $\bn$ is the size of its derivative, i.e., the integer 
\begin{equation}\label{eqn:weight(n)}
\omega(\bn) = \lvert\Delta \bn \rvert = \sum_{(i,j) \in \lambda}\Delta \bn({i,j}).
\end{equation}
\end{definition}

Both derivative and weight are additive with respect to the addition in $\BRPP^\lambda$.

The notion of weight gives rise to the following further characterisation of Young indicators.

\begin{lemma}\label{LemmaYoungIndicatorWeightOne}
A reverse plane partition
$\bn$ is a Young indicator if and only if $\omega(\bn) = 1$.
\end{lemma}
\begin{proof}
Suppose $\bn$ is a Young indicator, and let $U$ be the nonempty connected upper set such that $\bn=\chi_U$.
Let  $\boldit{d} = \Delta \bn$, 
and let $\mathrm{Min}(U) = \{(i_1, j_1), (i_2, j_2), \ldots, (i_t, j_t)\}$ denote
the set of minimal elements of $U$,
with $i_1 <\cdots < i_t$ and $j_1 >  \cdots > j_t$. 
For any $(i,j) \in \lambda$, 
it follows from the definitions of derivative and upper set that 
$-1 \leq \boldit{d}(i,j) \leq 1$,
that
 $\boldit{d}(i,j) = 1$ if and only if $(i,j) \in U$ while $(i-1,j),(i,j-1)\notin U$,
and that
$\boldit{d}(i,j) = -1$ if and only if  $(i-1,j),(i,j-1) \in U$ while $(i-1,j-1) \notin U$.
 In other words, 
 $\boldit{d}(i,j) = 1$ if and only if $(i,j) \in \mathrm{Min}(U)$,
 and 
$\boldit{d}(i,j) = 1$ if and only if $(i,j) \in V = 
\{(i_{2}, j_{1}), (i_{3}, j_{2}), \ldots, (i_{t}, j_{t-1})\}$.
By the previous discussion, 
we have $\omega(\bn)=\lvert\boldit{d}\rvert=\lvert\mathrm{Min}(U)\rvert-\lvert V\rvert =  t-(t-1) = 1$.

Conversely, suppose  $\omega(\bn)=1$. 
By  \Cref{PropositionAtoms},  additivity of $\omega$, and the first half of the proof it follows that $\bn$ must be a Young indicator.
%
\end{proof}

The following is the main result of this section.

\begin{theorem}\label{TheoremRPPHalfFactorial}
Let $\bn$ be a reverse plane partition. 
Then the length of any factorisation of $\bn$ is equal to the weight $\omega(\bn)$.
\end{theorem}

\begin{proof}
Combine \Cref{PropositionAtoms}, the additivity of $\omega$, and \Cref{LemmaYoungIndicatorWeightOne} with one another.
\end{proof}

\Cref{TheoremRPPHalfFactorial} says that $\BRPP^\lambda$ is a \emph{half-factorial} monoid, that is, even though factorisations are generally not unique, all factorisations of a given element have the same length. 
This notion originates in algebraic number theory,
and is of central importance in factorisation theory \cite{CC,GSOSN,GHK}, and in this paper precisely reflects, from a combinatorial perspective, the equidimensionality of $C^{[\bn]}$ (cf.~\Cref{thm:irr components}).

\begin{remark}
To the monoids $\BRPP^\lambda \subset \BN^\lambda$ we can associate the toric rings
$\BC[\BRPP^\lambda]$.
They were studied e.g.~in \cite{FR,Garsia} in the more general context of $P$-partitions.
In particular, they are always  normal Cohen--Macaulay rings.
Moreover, 
it follows from \Cref{TheoremRPPHalfFactorial} that the associated toric ideal, i.e., the ideal $I \subset \BC[\BN^\lambda]$ such that $\BC[\BRPP^\lambda]=\BC[\BN^\lambda]/I$, is homogeneous, 
and, therefore, it defines a projective toric variety.
\end{remark}

We conclude this section presenting a direct process for computing the weight of a reverse plane partition (cf.~\Cref{prop:weight-formula} for the formula).

\begin{definition}\label{def:socle-subsocle}
Let $\lambda$ be a partition. Its \emph{socle} $\Soc(\lambda)$ is the set of maximal elements of $\lambda$ with respect to the partial order, equivalently, 
\[
\Soc(\lambda) =\Set{ (i,j) \in \lambda \, | \, (i+1, j), (i,j+1) \notin \lambda}\subset\lambda.
\]
We define the \emph{subsocle} of $\lambda$ to be 
\[
\Subsoc(\lambda)
=\Set{ (i,j) \in \lambda \, | \,
(i+1, j), (i,j+1) \in \lambda, (i+1,j+1) \notin \lambda}\subset\lambda.
\]
\end{definition}
For example, in the partition $\lambda = (5,5,3,3,2,2)$ defined by the Young diagram 
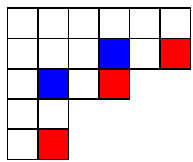
\begin{center}
\vspace{0.3cm}
\begin{tikzpicture}[scale=0.85]
\draw (0,0)--(1,0)--(1,1)--(2,1)--(2,1.5)--(3,1.5)--(3,2.5)--(0,2.5) -- cycle 
    (0,2 ) -- (3,2) (0,1.5 ) -- (2,1.5) (0,1 ) -- (1,1) (0,0.5 ) -- (1,0.5) (0.5,0)--(0.5,2.5) (1,1)--(1,2.5) (1.5,1)--(1.5,2.5) (2,1.5)--(2,2.5) (2.5,1.5)--(2.5,2.5);

    \draw[pattern=crosshatch] (0.5,0)--(1,0)--(1,0.5)--(0.5,0.5)--cycle;
    \draw[pattern=crosshatch] (1.5,1)--(2,1)--(2,1.5)--(1.5,1.5)--cycle;
    \draw[pattern=crosshatch] (2.5,1.5)--(3,1.5)--(3,2)--(2.5,2)--cycle;
    
    \draw[pattern=grid] (1.5,1.5)--(2,1.5)--(2,2)--(1.5,2)--cycle;
    
    \draw[pattern=grid] (0.5,1)--(1,1)--(1,1.5)--(0.5,1.5)--cycle;
\end{tikzpicture}
\end{center}
the \tikz[scale=0.5]{ 
    \draw[pattern=crosshatch] (0.5,0)--(1,0)--(1,0.5)--(0.5,0.5)--cycle; 
}-boxes $(1,4)$, $(3,2)$ and $(5,1)$ are the elements of the socle, whereas 
the \tikz[scale=0.5]{
    \draw[pattern=grid] (0.5,0)--(1,0)--(1,0.5)--(0.5,0.5)--cycle; 
}-boxes $(1,2)$ and $(3,1)$ are the elements of the subsocle.

We totally order the elements of the socle $\Soc(\lambda)$ by declaring $(i,j)\preceq (l,k)$ if $i\leq l$, for all $(i,j), (l,k)\in \Soc(\lambda)$, and we say that two socle elements are \emph{consecutive} if they are consecutive with respect to this total order.

\begin{prop}\label{prop:weight-formula}
Let $\bn$ be a reverse plane partition with shape $\lambda$.
Then
\[
\omega(\bn) = \sum_{(i,j)\in \Soc(\lambda)} \bn(i,j)\ - \sum_{(i,j)\in \Subsoc(\lambda)} \bn(i,j).
\]
\end{prop}

\begin{proof}
Both sides of the equation are additive with respect to the addition of $\BRPP^\lambda$,
thus, by \Cref{PropositionAtoms} and  \Cref{LemmaYoungIndicatorWeightOne},
it suffices to show that the right hand side is equal to one when $\bn$ is a Young indicator.
This is immediate, since for any nonempty connected upper set $U$ we have
$|U \cap \Soc(\lambda)| = 1 + |U \cap \Subsoc(\lambda)|$.
\end{proof}


We conclude this section by discussing some distinguished classes  of factorisations in the monoid $\BRPP^\lambda$.
We will see in Subsection \ref{sec: smoothness} that they geometrically correspond to smooth irreducible components of the double nested Hilbert scheme.

\begin{definition}\label{def: standard}
Let $\bn$ be a reverse plane partition of shape $\lambda$. 
Let $\bn(\lambda)=\{k_1 < \cdots < k_t\}\subset \BN$ be the set of values labelling  the boxes of $\lambda$,
and set $k_0=0$. For $s=1,\ldots,t$, consider the upper sets 
\[
V_s=\bn^{-1}(\{k_s,\dots, k_t\})\subset \lambda,
\]
and their unique decompositions
$
V_s=\coprod_{i=1}^{r_s}U_{s,i}$,
where each $U_{s,i}$ is a nonempty connected upper set. 
The \emph{standard factorisation}  of $\bn$ is  given by
\[
\bn = \sum_{s=1}^t\sum_{i=1}^{r_s} \,( k_s-k_{s-1})\chi_{U_{s,i}}.
\]
\end{definition}

Standard factorisations appeared in the  proof of Proposition \ref{PropositionAtoms}.
For each  element of $\BRPP^\lambda$,
the standard factorisation always exists  and is unique.

 \begin{definition}\label{def-standard-factorisation}
 Let $\bn $ be a reverse plane partition of shape $\lambda$.
 A factorisation $T$ of $\bn$ is \emph{complete} if,
 for every Young indicator $\chi_U$ appearing in $T$ (with nonzero coefficient), the upper set $U \subseteq \lambda$ has a unique minimum with respect to the partial order of $\lambda$.
\end{definition}

\begin{example}
In Example \ref{Example3Factorisations},     
the factorisation
\begin{center}
    \begin{tikzpicture}
    \node at (-1.25,0) {$\bn$};
    \node at (0,0) {\young(01,11)};
    \node at (-0.75,0) {$= 2\cdot$};
    \node at (1.55,0) {\young(00,01)};
    \node at (0.8,0) {$+ 2\cdot$};
    \end{tikzpicture}
\end{center}
is the standard factorisation of $\bn$, whereas
\begin{center}
    \begin{tikzpicture}
    \node at (-1.25,-1.2) {$\bn$};
    \node at (0,-1.2) {\young(01,01)};
    \node at (-0.75,-1.2) {$= 2\cdot$};
    \node at (1.55,-1.2) {\young(00,11)};
    \node at (0.8,-1.2) {$+ 2\cdot$};   
    \end{tikzpicture}
\end{center}
is a complete factorisation.
\end{example}

Unlike the standard factorisation, complete factorisations of a given reverse plane partition $\bn$ may not exist: for   example, take any Young indicator which is not complete as a factorisation.
However, we will see that if a complete factorisation exists, it is unique. 

\begin{prop}\label{prop-complete}
Let $\lambda$ be a partition. A reverse plane partition $\bn \in \BRPP^\lambda$ admits a complete factorisation if and only if $\Delta \bn$ is a nonnegative Young filling.
\end{prop}
\begin{proof}
Observe that, if $U\subseteq \lambda$ is an upper set with a unique minimum, then the derivative of its Young indicator has value 1 at the unique minimum, and 0 elsewhere. 
In particular, $\Delta \chi_U$ is a nonnegative Young filling.

If $\bn$ admits a complete factorisation, then, by additivity of $\Delta$, 
it follows that $\Delta \bn$ is also a nonnegative Young filling.

Conversely, assume that $\Delta \bn$ is nonnegative. 
For each $(i,j)\in \lambda$, denote by $U_{(i,j)}\subseteq \lambda$  the  nonempty connected upper set with $(i,j)$ as unique minimum. 
Then, we have
\begin{equation}\label{eq-complete}
\bn = \sum_{(i,j) \in \lambda} \Delta \bn(i,j)\cdot \chi_{U_{(i,j)}},
\end{equation}
and this implies the desired conclusion.
Indeed, again by additivity of $\Delta$, 
 both sides of \eqref{eq-complete} have the same derivative, and
this forces the two reverse plane partitions to coincide.
\end{proof}

\begin{corollary}\label{cor:uniquecomplete}
A reverse plane partition $\bn$ admits at most one  complete factorisation.
\end{corollary}
\begin{proof}
With notation as in Proposition \ref{prop-complete}, 
let 
$\bn = \sum_{(i,j) \in \lambda} m_{i,j} \chi_{U_{(i,j)}}$ be a complete factorisation.
Then, again by additivity of $\Delta$, it follows that $\Delta \bn (i,j)=m_{i,j}$ for all $(i,j) \in \lambda$.
Thus, a complete factorisation is uniquely determined by $\Delta \bn$.
\end{proof}

\begin{notation}
In the rest of the paper, for the sake of exposition, given a box $\Box=(i,j) \in \lambda$, we will denote by $\bn_{\Box} \in \BN$ or by $\bn_{i,j}$ the nonnegative integer $\bn(\Box) = \bn(i,j)$.
\end{notation}

\section{Components: dimension,  smoothness and classification}

\subsection{The moduli space and its functor of points}

Let $X$ be a $\bfk$-variety. Fix a Young diagram $\lambda$ and a reverse plane partition $\bn = (\bn_{\Box})_{\Box \in \lambda}\in\BRPP^\lambda$ with shape $\lambda$. We now formally introduce the moduli space $X^{[\bn]}$ which is the object of study of this paper.

\begin{definition}\label{def:config}
Let $B$ be a $\bfk$-scheme. A $B$-\emph{family of configurations of $X$}, with underlying reverse plane partition $\bn\in \BRPP^\lambda$, is a $\lambda$-tuple $\CZ=(\CZ_{\Box} \into X \times B)_{\Box\in\lambda}$ of closed subschemes of $X \times B$, each flat and finite of relative length $\bn_{\Box}$ over $B$, and equipped with a closed immersion $\CZ_\Box \into \CZ_{\Box'}$ over $B$ for all $\Box,\Box'\in\lambda$ such that $\Box \leq \Box'$. In this context, we will say that $\CZ$ has \emph{shape} $\lambda$. 
\end{definition}

We often omit the closed immersions into $X \times B$ and, by a slight abuse of notation, we just write $(\CZ_{\Box})_{\Box \in \lambda}$ instead of $(\CZ_{\Box} \into X \times B)_{\Box\in\lambda}$.

A $B$-family of configurations of $X$ can be visualised as a diagram of nested subschemes
\begin{equation*}
  \begin{tikzcd}
    \CZ_{0,0}\arrow[r,hook]\arrow[d,hook] &\CZ_{1,0}\arrow[d,hook]\arrow[r,hook]&\CZ_{2,0}\arrow[d,hook]\arrow[r,hook]\arrow[d,hook]&\CZ_{3,0}\arrow[r,hook]\arrow[d,hook]&\cdots\\      \CZ_{0,1}\arrow[r,hook]\arrow[d,hook]&\CZ_{1,1}\arrow[d,hook]\arrow[r,hook]&\CZ_{2,1}\arrow[d,hook]\arrow[r,hook]&\CZ_{3,1}\arrow[d,hook]\arrow[r,hook]&\cdots\\
        \CZ_{0,2}\arrow[d,hook]\arrow[r,hook]&\CZ_{1,2}\arrow[d,hook]\arrow[r,hook]&\CZ_{2,2}\arrow[d,hook]\arrow[r,hook]& \rotatebox[origin=c]{-45}{$\cdots$} &\\
          \vdots&\vdots&\vdots& &\\
  \end{tikzcd}  
\end{equation*}
of $X \times B$, with the same shape as $\lambda$.

Consider the moduli functor 
\[
\mathsf{Hilb}^{\bn}(X) \colon \Sch_{\bfk}^{\opp} \to \Sets,
\]
defined by sending a $\bfk$-scheme $B$ to the set of $B$-families of configurations of $X$ with underlying reverse plane partition $\bn\in \BRPP^\lambda$. If $X$ is quasiprojective, then  the moduli functor $\mathsf{Hilb}^{\bn}(X)$ is represented by a quasiprojective $\bfk$-scheme, which we denote by $X^{[\bn]}$, see \cite[Prop. 2.4]{Mon_double_nested} (cf. also \cite{Ser_deformation} for the case of nested Hilbert schemes). We call $X^{[\bn]}$ the \textit{double nested Hilbert scheme of points on $X$}.

\begin{notation}
Given a configuration $Z$ of $X$ with reverse plane partition $\bn $, we will denote by $[Z]$ the corresponding point of $X^{[\bn]}$.
\end{notation}

In the remaining part of this section  we study the geometry of $X^{[\bn]}$ in the case where $X=C$ is a smooth curve (in particular, connected). In this situation, the $0$-dimensional subschemes of $C$ are precisely the effective $0$-cycles. 

\subsection{Irreducible components, purity and connectedness of \texorpdfstring{$C^{[\bn]}$}{C\^n}}\label{sec: irreducible components}
Let $C$ be a smooth curve. The moduli space $ C^{[\bn]}$ is in general reducible and, thus, not smooth, as it was already observed by the third author in \cite{Mon_double_nested}. In this subsection, we establish a bijection between the irreducible components of $C^{[\bn]}$ and the set $\Fact(\bn)$ of factorisations of $\bn$. Moreover, we prove that $C^{[\bn]}$ is pure of dimension $\omega(\bn)$. The notion of \emph{weight} $\omega(-)$ was defined in \eqref{eqn:weight(n)} (cf.~also \Cref{prop:weight-formula}).

\begin{construction}\label{construction:linear-comb-cycles}
Fix a smooth curve $C$, a partition $\lambda$ and a $\bfk$-scheme $B$. Let  $\bn_1,\ldots,\bn_s\in\BRPP^\lambda$ be reverse plane partitions and let $\CZ_1,\ldots,\CZ_s\subset C \times B$ be $B$-flat families of $0$-dimensional subschemes  of $C$. Setting
\[ 
\left(\sum_{i=1}^s \bn_i\cdot \CZ_i\right)_{\Box}=\sum_{i=1}^s \bn_{i,\Box}\cdot \CZ_i
\]
defines a $B$-flat family of configurations of $C$, where the right hand side is defined thanks to functoriality of the sum of $0$-cycles map (see \cite[Paper 1, p.~40]{Rydh1} for the precise statement in a very general setup).
\end{construction}

\begin{definition}\label{def:free-conf}
Fix a smooth curve $C$, a partition $\lambda$, a reverse plane partition $\bn \in \BRPP^\lambda$ and a $\bfk$-scheme $B$. Let $\CZ = (\CZ_{\Box} \into C \times B)_{\Box \in \lambda}$ be a $B$-valued point of $C^{[\bn]}$. We say that $\CZ$ is \emph{free} if there exist (not necessarily distinct) indicators $\bnu_1,\ldots,\bnu_{\omega(\bn)}\in\Ind(\lambda)$ and pairwise disjoint sections $\sigma_1,\ldots,\sigma_{\omega(\bn) } \colon B \to C \times B$ of the projection $C \times B \to B$ such that 
\[
\CZ = \sum_{i=1}^{\omega(\bn)} \bnu_i \cdot \sigma_i(B).
\]
\end{definition}

If $B = \Spec \bfk$, \Cref{def:free-conf} reads simply as follows. A configuration $Z = (Z_{\Box} \into C)_{\Box \in \lambda}$ of a curve $C$ is \textit{free} if there exist (not necessarily distinct) indicators $\bnu_1,\ldots,\bnu_{\omega(\bn)}\in\Ind(\lambda)$ and distinct closed points $P_1,\ldots,P_{\omega(\bn)} \in C$ such that 
\begin{equation}
\label{eq:smooth}Z=\sum_{i=1}^{\omega(\bn)} \bnu_i\cdot P_i.
\end{equation}

\begin{remark}\label{rem:unique smooth}
Given any free configuration,  the  associated collection of  indicators $ \set{\bnu_i}_{i=1, \ldots,{\omega(\bn)}}$ is a factorisation of the reverse plane partition $\sum_{i=1}^{\omega(\bn)} \bnu_i$. Notice that, since indicators are irreducible objects in the monoid of reverse plane partitions of fixed length, any free configuration  can be written in the form \eqref{eq:smooth} in a unique way. 
\end{remark}

\begin{lemma}\label{lemma:smoothisopendense} 
Let $C$ be a smooth curve. Let $\lambda$ be a partition, and fix $\bn \in \BRPP^\lambda$. The locus $C^{[\bn]}_{\free} \subset C^{[\bn]}$ consisting of free configurations of $C$ is open and dense.
\end{lemma}
\begin{proof}
First, notice that, as a consequence of \Cref{PropositionAtoms} and \Cref{TheoremRPPHalfFactorial}, all configurations $Z$ with reverse plane partition $\bn$ can be written in the form 
\begin{equation} \label{eqn:arbitrary-Z-dec}
Z = \sum_{i=1}^{\omega(\bn)} \bnu_i\cdot  P_i,
\end{equation}
namely, in the form \eqref{eq:smooth}, but where the points $P_i$ are not necessarily distinct. The locus $C^{[\bn]}_{\free}$ is then open because its complement is defined by closed conditions. Namely, the `collision' of two points in \eqref{eq:smooth}. To confirm density, observe that any such $Z$ is a limit of configurations of the form $\sum_{i=1}^{\omega(\bn)} \bnu_i\cdot  P_{i,t} $, where the points $P_{i,t}$ are pairwise distinct for $t\neq 0$.
\end{proof}

We now move towards establishing a bijective correspondence between the set of irreducible components of $C^{[\bn]}$ and the set of factorisations of $\bn$ (this is \Cref{thm:irr components}). We provide \Cref{ex:quadrato} in order to anticipate how the correspondence works.

\begin{example}\label{ex:quadrato}
    It was shown in \cite[Ex. 2.6]{Mon_double_nested} that the double nested Hilbert scheme $C^{[\bn]}$ defined by the reverse plane partition
    \[
\begin{tikzpicture}
    \node at (0,0) {\young(01,12)};
    \node at (-0.85,-0.03) {$\bn=$};
\end{tikzpicture}
    \]
has two irreducible components $V_1$ and $V_2$, both of dimension 2. Two general points $[Z_i]\in V_i$, for $i=1,2$, correspond to the free configurations
    \[
Z_1=  \begin{tikzcd}[column sep =tiny,row sep =tiny]
    \varnothing\arrow[r, phantom, "\subset"]\arrow[d, phantom, "\cap"] & P\arrow[d, phantom, "\cap"]\\
      P\arrow[r, phantom, "\subset"]  &P+Q 
  \end{tikzcd}
  \qquad
  \text{and}
  \qquad
  Z_2=\begin{tikzcd}[column sep =tiny,row sep =tiny]
    \varnothing\arrow[r, phantom, "\subset"]\arrow[d, phantom, "\cap"] &Q\arrow[d, phantom, "\cap"]\\
      P\arrow[r, phantom, "\subset"]  &P+Q 
  \end{tikzcd}  
    \]
where $P,Q\in C$ are distinct closed points. Now, $\bn$ admits two possible factorisations
\begin{center}
    \begin{tikzpicture}
    \node at (-1.1,0) {$\bn$};
    \node at (0,0) {\young(01,11)};
    \node at (-0.7,-0.03) {$= $};
    \node at (1.55,0) {\young(00,01)};
    \node at (0.8,0) {$+ $};
    
    \node at (0,-1.2) {\young(01,01)};
    \node at (-0.75,-1.23) {$= $};
    \node at (1.55,-1.2) {\young(00,11)};
    \node at (0.8,-1.2) {$+ $};
    \end{tikzpicture}
\end{center}
and the formal expressions
\begin{center}
    \begin{tikzpicture}
    \node at (-1.1,0) {$Z_1$};
    \node at (0,0) {\young(01,11)};
    \node at (-0.7,-0.03) {$= $};
    \node at (1.7,0) {\young(00,01)};
    \node at (0.8,0) {$ \cdot P\ + $};
    \node at (2.35,0) {$ \cdot Q $};
    \end{tikzpicture}\\
    \begin{tikzpicture}
    \node at (-1.1,0) {$Z_2$};
    \node at (0,0) {\young(01,01)};
    \node at (-0.7,-0.03) {$= $};
    \node at (1.7,0) {\young(00,11)};
    \node at (0.8,0) {$ \cdot Q\ + $};
    \node at (2.35,0) {$ \cdot P $};
    \end{tikzpicture}
\end{center}
give a canonical way to associate bijectively a factorisation of $\bn $ to each irreducible component of $C^{[\bn]}$. 
\end{example}

\begin{prop}\label{prop:conn-comp-free}
Let $C$ be a smooth curve, $\lambda$ a partition. Fix $\bn \in \BRPP^\lambda$ and a factorisation $T=(n_{\bnu})_{\bnu \in \Ind(\lambda)}$ of $\bn$. Consider the open subscheme
\[
\begin{tikzcd}
C_T\arrow[hook]{r} &  \displaystyle\prod_{\bnu\in\Ind(\lambda)}C^{(n_{\bnu})},
\end{tikzcd}
\]
consisting of tuples $(D_{\bnu})_{\bnu \in \Ind(\lambda)}$ of reduced divisors $D_{\bnu} \subset C$ with pairwise disjoint supports. We have an isomorphism of schemes
\[
\begin{tikzcd}
C^{[\bn]}_{\free}\arrow{r}{\sim} &\displaystyle\coprod_{T\in \Fact(\bn)} C_T.
\end{tikzcd}
\]
In particular, $C^{[\bn]}_{\free}$ is smooth, with connected components indexed by the factorisations of $\bn$.
\end{prop}

\begin{proof}
We prove the isomorphism by comparing the associated moduli functors. We first prove that there is a bijection at the level of closed points. Let $Z$ be a closed point of $C^{[\bn]}_{\free}$. Then $Z$ is a free configuration and, following \Cref{rem:unique smooth}, it can be written uniquely as 
\[Z=\sum_{i=1}^{\omega(\bn)} \bnu_i\cdot P_i,\]
for some distinct closed points $\{P_i\}_i$, with associated factorisation $T$. Given an indicator $\bnu$,  denote by $n_{\bnu}$ the multiplicity of the indicator $\bnu$ in the factorisation $T$, and by $P_{\bnu,1},\dots, P_{\bnu, n_{\bnu}}$ the distinct point in $Z$ with coefficient $\bnu$. To $Z$, we associate the closed point
\[
p_Z = \left(\sum_{j=1}^{n_{\bnu}}P_{\bnu, j} \right)_{\bnu\in \Ind(\lambda)}\in C_T.
\]
It is clear that the association $Z \mapsto p_Z$ just defined gives a bijection between the closed points of $C^{[\bn]}_{\free}$ and the closed points of $\coprod_{T\in \Fact(\bn)} C_T$. The same argument naturally carries over to flat families, therefore, we obtain the sought after isomorphism.
\end{proof}

\begin{corollary}\label{cor:gen-red}
Let $C$ be a smooth curve. Let $\lambda$ be a partition, and fix $\bn \in \BRPP^\lambda$.
The scheme $C^{[\bn]}$ is generically reduced.
\end{corollary}

\begin{proof}
It is enough to notice that $C^{[\bn]}$ contains the dense open subset $C^{[\bn]}_{\free}$, which is smooth by \Cref{prop:conn-comp-free}, and, therefore, reduced.
\end{proof}

The conclusion of \Cref{ex:quadrato} is generalised in the following result.

\begin{theorem}\label{thm:irr components} Let $C$ be a smooth curve. Let $\lambda$ be a partition, and fix $\bn \in \BRPP^\lambda$. Then
\begin{itemize}
    \item [\mylabel{irr-cpt-1}{\normalfont{(1)}}] there is a bijection
    \[
    \begin{tikzcd}
    \Set{
    \mbox{\normalfont{irreducible components} }V \subset C^{[\bn]}
    }
    \arrow[r,"\sim"] &\Fact(\bn),
    \end{tikzcd} 
    \]
    \item [\mylabel{irr-cpt-2}{\normalfont{(2)}}] all the irreducible components of $C^{[\bn]}$ have dimension equal to $\omega(\bn)$, i.e., $C^{[\bn]}$ is pure with
    \[
    \dim C^{[\bn]} = \omega(\bn).
    \]
\end{itemize}
\end{theorem}

\begin{proof}
The irreducible components of $C^{[\bn]}$ are in bijection with the irreducible components of any dense open subset, for instance $C^{[\bn]}_{\free}$ (cf.~\Cref{lemma:smoothisopendense}). But the irreducible components of $C^{[\bn]}_{\free}$ are precisely the closures of its connected components. Thus \ref{irr-cpt-1} is proved. 

As for \ref{irr-cpt-2}, we observe that for every $T\in\Fact(\bn)$, one has
\[
\dim \overline{C_T} = \dim C_T = \sum_{\bnu \in \Ind(\lambda)}n_{\bnu} = \omega(\bn),
\]
where the last equality is the content of \Cref{TheoremRPPHalfFactorial}.
\end{proof}

See \Cref{prop: dimension} for a direct proof that $\dim C^{[\bn]} = \omega(\bn)$, including an explicit example (cf.~\Cref{example:dimension}), showing  the algorithmic nature of this alternative proof.

\begin{corollary}\label{cor:conn}
Let $C$ be a smooth curve. Let $\lambda$ be a partition, and fix $\bn \in \BRPP^\lambda$. Then the intersection of  all irreducible components of $C^{[\bn]}$ contains a copy of $C$. In particular, $C^{[\bn]}$ is connected.
\end{corollary}

\begin{proof}
The locus $\overline{C_T}$ contains a copy of $C$ for every $T \in \Fact(\bn)$, where the closure is taken in $C^{[\bn]}$.
Indeed, a general point of $\overline{C_T}$ can be degenerated to a fully punctual double nested scheme (i.e.~one where all the divisors $Z_{\Box}\into C$ are entirely supported  at the same closed point of $C$) by making the supporting points collide together to a fixed one.
\end{proof}

\subsection{Singularities of the irreducible components}\label{sec: smoothness}
Since $C^{[\bn]}$ is always connected (cf.~\Cref{cor:conn} or \Cref{prop:connected}), it is singular as soon as it is reducible, i.e., as soon as $\lvert \Fact(\bn)\rvert > 1$. One may still ask whether the irreducible components of $C^{[\bn]}$ are smooth. We conclude this subsection with a negative answer to this question, and provide combinatorial criteria for smoothness.

\begin{prop}\label{prop:canonical cover}
Let $C$ be a smooth curve. Let $\lambda$ be a partition, and fix $\bn \in \BRPP^\lambda$. Let $V\subset C^{[\bn]}$ be an irreducible component, with associated factorisation $(n_{\bnu})_{\bnu \in \Ind(\lambda)} \in \Fact(\bn)$. There is a finite birational morphism 
\[
\begin{tikzcd}
\displaystyle\prod_{\bnu \in \Ind(\lambda)} C^{(n_{\bnu})} \arrow{r}{\varphi_V} & V,
\end{tikzcd}
\]
which coincides with the normalisation morphism.
\end{prop}

\begin{proof}
We have a morphism of schemes
\[
\begin{tikzcd}
\displaystyle\prod_{\bnu \in \Ind(\lambda)} C^{(n_{\bnu})}\arrow[r,"\varphi_V"]& C^{[\bn]}
\end{tikzcd}
\]
arising directly from \Cref{construction:linear-comb-cycles}. On points, it is given by 
\[
\begin{tikzcd}
\left(D_{\bnu}\right)_{\bnu \in \Ind(\lambda)} \arrow[mapsto]{r} & \displaystyle\sum_{\bnu \in \Ind(\lambda)} \bnu \cdot D_{\bnu} =  
\left(\displaystyle\sum_{\bnu \in \Ind(\lambda)}\bnu_{\Box} \cdot D_{\bnu}\right)_{\Box \in \lambda}.
\end{tikzcd}
\]
The image of $\varphi_V$ is irreducible (since its source is), and its restriction to the open subscheme $C_T$ of the source is an open immersion (cf.~\Cref{lemma:smoothisopendense} and \Cref{prop:conn-comp-free}). Thus, the image of $\varphi_V$ is an irreducible component of $C^{[\bn]}$, necessarily equal to $V$. This argument proves the existence of the morphism
\[
\begin{tikzcd}
\displaystyle\prod_{\bnu \in \Ind(\lambda)} C^{(n_{\bnu})}\arrow[r,"\varphi_V"]& V
\end{tikzcd}
\]
and its birationality at once.

To show that $\varphi_V$ is finite (proper and quasifinite), we argue as follows. Let $\overline C$ be a smooth compactification of $C$. The chosen factorisation corresponds, by \Cref{thm:irr components}, to an irreducible component $\overline{V} \subset \overline C^{[\bn]}$. Now, we have a commutative diagram 
\[
\begin{tikzcd}[row sep=large,column sep=large]
\displaystyle\prod_{\bnu \in \Ind(\lambda)} C^{(n_{\bnu})}\arrow[r,"\varphi_{V}"]\arrow[d,hook]&V\arrow[d,hook]\\
\displaystyle\prod_{\bnu \in \Ind(\lambda)} \overline{C}^{(n_{\bnu})}\arrow[r,"\varphi_{\overline{V}}"]&\overline{V}
\end{tikzcd}
\]
which is \emph{cartesian}, thus, since finiteness is stable under base change, it is enough to prove it in the projective case. In this case, properness is automatic (since the source and target of $\varphi_{\overline{V}}$ are both proper over $\bfk$). Quasifiniteness follows directly by checking that each fibre of $\varphi_V$ is finite. Finally, the fact that $\varphi_V$ coincides with the normalisation morphism is a consequence of \cite[§4.1.2, Prop.~1.22 and 1.25]{Liu_AG}.
\end{proof}

In the remainder of this section, we give criteria to understand when the map $\varphi_V$ is an isomorphism, and how to detect smoothness of $V$ purely from the combinatorics of the associated factorisation (cf.~\Cref{thm: main thm for criterion of smooth}). In particular, we will see that as soon as an irreducible component is not smooth, it is automatically not normal (cf. \Cref{thm:smoothness-components}). We first illustrate this by means of an example.

\begin{example}\label{ex:singular}
Let $C$ be a smooth curve and consider the reverse plane partition
\[ 
\begin{tikzpicture}
    \node at (0,0) {\young(02,24)};
    \node at (-0.85,-0.03) {$\bn=$};
\end{tikzpicture}
\]
of \Cref{Example3Factorisations},
    and its factorisation 
    \begin{center}
    \begin{tikzpicture}
    \node at (-1.2,-2.45) {$\bn$};
    \node at (0,-2.4) {\young(01,11)};
    \node at (-0.75,-2.45) {$= $};
    \node at (1.55,-2.4) {\young(00,01)};
    \node at (0.77,-2.4) {$+$};
    \node at (3.1,-2.4) {\young(01,01)};
    \node at (2.32,-2.4) {$+$};
    \node at (3.87,-2.4) {$+$};
    \node at (4.65,-2.4) {\young(00,11).};
    \end{tikzpicture}
    \end{center}
We show that the irreducible component $V\subset C^{[\bn]}$ corresponding to this factorisation is not normal, hence, in particular, not smooth.
    Notice that there is a relation
    \begin{center}
    \begin{tikzpicture}
    \node at (0,-2.4) {\young(01,11)};
    \node at (1.55,-2.4) {\young(00,01)};
    \node at (0.77,-2.4) {$+$};
    \node at (3.1,-2.4) {\young(01,01)};
    \node at (2.32,-2.4) {$=$};
    \node at (3.87,-2.4) {$+$};
    \node at (4.65,-2.4) {\young(00,11)};
    \end{tikzpicture}
    \end{center}
between the indicators appearing in this factorisation. Relations such as this one are precisely those preventing the smoothness  of the irreducible components of the moduli space, as will be formalised in \Cref{thm: main thm for criterion of smooth}. Let $P, Q$ be distinct closed points in $C$ and consider the configuration
\[
Z=\begin{matrix}
            \varnothing&\hspace{-10pt}\subset&\hspace{-10pt} P+Q\\
            \cap &&\hspace{-10pt}\cap\\
            P+Q&\hspace{-10pt}\subset&\hspace{-10pt} 2P+2Q.
        \end{matrix}
    \]
The preimage of $[Z] \in V$ under $\varphi_V$ consists of two points, namely
\[
\varphi_V^{-1}([Z])=\set{(P,P, Q, Q), (Q,Q,P,P)}\subset C\times C\times C\times C.
\]
Thus, $\varphi_V$ is not an isomorphism. Therefore, $V$ is not normal by \cite[\href{https://stacks.math.columbia.edu/tag/0AB1}{Tag 0AB1}]{stacks-project}.
\end{example}

The key point in \Cref{ex:singular}, in order to deduce the singularity of an irreducible component, was to exhibit a relation among the indicators appearing in the corresponding factorisation. We promote this idea to a general criterion for smoothness. We start proving some preliminary results.

\begin{lemma}\label{thm:smoothness-components}
For every irreducible component $V\subset C^{[\bn]}$, the following conditions are equivalent:
\begin{itemize}
    \item [\mylabel{equiv-cond-1}{\normalfont{(1)}}] $V$ is smooth,
    \item [\mylabel{equiv-cond-2}{\normalfont{(2)}}] $V$ is normal,
    \item [\mylabel{equiv-cond-3}{\normalfont{(3)}}] $\varphi_V$ is an isomorphism,
    \item [\mylabel{equiv-cond-4}{\normalfont{(4)}}] $\varphi_V$ a bijection on points and its differential is injective.
\end{itemize}
\end{lemma}
\begin{proof}
Implications \ref{equiv-cond-1} $\Rightarrow$ \ref{equiv-cond-2}, \ref{equiv-cond-3} $\Rightarrow$ \ref{equiv-cond-1} and \ref{equiv-cond-3} $\Rightarrow$ \ref{equiv-cond-4} are obvious. Implication \ref{equiv-cond-2} $\Rightarrow$ \ref{equiv-cond-3} follows from \cite[\href{https://stacks.math.columbia.edu/tag/0AB1}{Tag 0AB1}]{stacks-project}, which says precisely that given two integral schemes $X$ and $Y$, with $Y$ normal, any finite birational morphism $f\colon X\to Y$ is an isomorphism. The morphism $\varphi_V$ satisfies the required assumptions by \Cref{prop:canonical cover}. To prove \ref{equiv-cond-4} $\Rightarrow$ \ref{equiv-cond-3} we observe that a bijective, finite morphism with injective differential  is, in particular, a bijective closed immersion. Source and target are integral and have the same dimension, thus, $\varphi_V$ is an isomorphism. 
\end{proof}
\begin{prop}\label{prop: bijec comb}
Let $V\subset C^{[\bn]}$ be an irreducible component, corresponding to a factorisation $T=(n_{\bnu})_{\bnu \in \Ind(\lambda)}$. Then, $\varphi_V$ is a bijection on points if and only if there is no relation
\begin{equation}\label{eq:singrelation}
\sum_{\bnu\in \Ind(\lambda)}m_{\bnu}\cdot\bnu=0,
\end{equation}
where $m_{\bnu}\in \BZ$ are integers such that
\begin{itemize}
\item [\mylabel{singular-cpt-i}{\normalfont{(i)}}]  $\lvert m_{\bnu }\rvert \leq n_{\bnu }$ for all $\bnu \in \Ind (\lambda)$, and
 \item [\mylabel{singular-cpt-ii}{\normalfont{(ii)}}] $m_{\bnu}\not = 0$ for at least one indicator $\bnu\in\Ind(\lambda)$.
\end{itemize}
\end{prop}
\begin{proof}
Suppose  that relation (\ref{eq:singrelation}) holds, along with conditions \ref{singular-cpt-i} and \ref{singular-cpt-ii}. Then we have the relation
\[
\sum_{\underset{\mbox{\tiny $\bnu \in \Ind(\lambda)$}}{\mbox{\tiny $m_{\bnu}\ge 0$}}} {m_{\bnu}} \cdot\bnu =
\sum_{\underset{\mbox{\tiny $\bnu \in \Ind(\lambda)$}}{\mbox{\tiny $m_{\bnu}< 0$}}} -{m_{\bnu}}\cdot\bnu.
\]
Fix three distinct points $P,Q,R\in C$ and define the function $\delta(m)=\max\{m,0\}$ for $m\in \BZ$.
  Consider the two points
\begin{align*}
    W_1&=\left(\delta(m_{\bnu})\cdot P+\delta(-m_{\bnu})\cdot Q +(n_{ \bnu}-|m_{\bnu}|)R    \right)_{\bnu\in \Ind(\lambda)}\in \prod_{\bnu \in \Ind(\lambda)} C^{(n_{\bnu})},\\
      W_2&=\left(\delta(m_{\bnu})\cdot Q+\delta(-m_{\bnu})\cdot P +(n_{ \bnu}-|m_{\bnu}|)R    \right)_{\bnu\in \Ind(\lambda)}\in \prod_{\bnu \in \Ind(\lambda)} C^{(n_{\bnu})}.
\end{align*}
By condition \ref{singular-cpt-ii}, the points $W_1, W_2$ are distinct. However, they map to the same point in $C^{[\bn]}$ via $\varphi_V$, thanks to the equality
\begin{multline*}
\sum_{\underset{\mbox{\tiny $\bnu \in \Ind(\lambda)$}}{\mbox{\tiny $m_{\bnu}\ge 0$}}} {m_{\bnu}} \cdot\bnu \cdot P +
\sum_{\underset{\mbox{\tiny $\bnu \in \Ind(\lambda)$}}{\mbox{\tiny $m_{\bnu}< 0$}}} (-{m_{\bnu}})\cdot\bnu \cdot Q +
\sum_{\bnu \in \Ind(\lambda)} (n_{\bnu}-{|m_{\bnu}|})\cdot\bnu \cdot R \\
=\sum_{\underset{\mbox{\tiny $\bnu \in \Ind(\lambda)$}}{\mbox{\tiny $m_{\bnu}\ge 0$}}} {m_{\bnu}} \cdot\bnu \cdot Q +
\sum_{\underset{\mbox{\tiny $\bnu \in \Ind(\lambda)$}}{\mbox{\tiny $m_{\bnu}< 0$}}} (-{m_{\bnu}})\cdot\bnu \cdot P +
\sum_{\bnu \in \Ind(\lambda)} (n_{\bnu}-{|m_{\bnu}|})\cdot\bnu \cdot R,
\end{multline*}
which implies that $\varphi_V$ is not a bijection.
Conversely, suppose that $\varphi_V$ is not a bijection. Let
\[
\left(\sum_{P\in C}a_{{\bnu}, P}\cdot P\right)_{\bnu\in \Ind(\lambda)}, \left(\sum_{P\in C}b_{{\bnu}, P}\cdot P\right)_{\bnu\in \Ind(\lambda)}\in \prod_{{\bnu} \in \Ind(\lambda)} C^{(n_{\bnu})} 
\]
be distinct points mapping to the same point in $C^{[\bn]}$ via $\varphi_V$, which implies that
\[
\sum_{\bnu \in\Ind(\lambda)}\sum_{\underset{a_{\bnu, P}\leq b_{\bnu, P}}{P\in C}}(b_{\bnu, P}-a_{\bnu, P})\bnu\cdot P=\sum_{\bnu \in\Ind(\lambda)}\sum_{\underset{a_{\bnu, P}> b_{\bnu, P}}{P\in C}}(a_{\bnu, P}-b_{\bnu, P})\bnu\cdot P.
\]
In particular, since the original points were taken to be distinct, there exists an indicator $\bnu\in \Ind(\lambda)$ and a point $P\in C$ such that $ a_{\bnu, P}\neq b_{\bnu, P}$.  If we fix such a $P\in C$, we deduce a nontrivial relation 
\[
\sum_{\underset{a_{\bnu, P}\leq b_{\bnu, P}}{\bnu \in\Ind(\lambda)}}(b_{\bnu, P}-a_{\bnu, P})\bnu=\sum_{\underset{a_{\bnu, P}> b_{\bnu, P}}{\bnu \in\Ind(\lambda)}}(a_{\bnu, P}-b_{\bnu, P})\bnu.
\]
Given an indicator $\bnu\in\Ind(\lambda)$, we set $m_{\bnu}=b_{\bnu, P}-a_{\bnu, P}$. The collection $(m_{\bnu})_{\bnu\in\Ind(\lambda)}$ clearly satisfies the relation \eqref{eq:singrelation} and condition \ref{singular-cpt-ii}. To show that condition \ref{singular-cpt-i} holds, suppose without loss of generality that $ b_{\bnu, P}\geq a_{\bnu, P}$. Then 
\begin{align*}
    n_{\bnu}=\sum_{P\in C} b_{\bnu,P}\geq b_{\bnu, P}\geq b_{\bnu, P}-a_{\bnu, P}=m_{\bnu},
\end{align*}
which concludes the proof.
\end{proof}

\begin{prop}\label{prop: inj diff comb}
Let $V\subset C^{[\bn]}$ be an irreducible component, corresponding to a factorisation $T=(n_{\bnu})_{\bnu \in \Ind(\lambda)}$. Then $\varphi_V$ has injective differential if and only if there is no nontrivial relation
    \begin{equation*}
        \sum_{\bnu\in \Ind(\lambda)}m_{\bnu}\cdot\bnu=0,
    \end{equation*}
where $m_{\bnu}\in \BZ$ are integers and are possibly nonzero only for the indicators $\bnu$ appearing in the factorisation $T$.
\end{prop}

\begin{proof}
       The statement is étale-local in nature so we suppose $C=\BA^1$. In this case, the morphism $\varphi_V$ is a torus-equivariant morphism of toric varieties, and therefore it is enough to check the injectivity of the  differential   at the unique torus-fixed point $[Y]=(n_{\bnu}\cdot 0)_{\bnu\in\Ind(\lambda)}$. 
    A tangent vector $v\in T_{[Y]} \prod_{\bnu \in\Ind(\lambda)}C^{(n_{\bnu})}$ can be represented by a tuple of ideals in $\bfk[x,\varepsilon]/(\varepsilon^2)$
    \[
    v=\left( x^{n_{\bnu}} + \varepsilon\sum_{i=1}^{n_{\bnu}} \alpha_{\bnu,i}  x^{n_{\bnu}-i} \right)_{\bnu\in\Ind(\lambda)},
    \]
    for some $\alpha_{\bnu,i}\in \bfk$  and  $i=1,\ldots,n_{\bnu}$.  In particular, the zero vector corresponds to choosing $a_{\bnu, i}=0$ for all $\bnu$ and $i$.
    Similarly, the image $\mathrm{d}\varphi_{V, [Y]}(v)$ can be represented by the tuple of ideals 
    \[
    \begin{split}
       \left( \prod_{\bnu\in\Ind(\lambda)}\left( x^{n_{\bnu}} + \varepsilon\sum_{i=1}^{n_{\bnu}} \alpha_{\bnu,i}  x^{n_{\bnu}-i} \right)^{\bnu_{\Box}}\right)_{\Box\in \lambda}&{}=\left(x^{\bn_{\Box} }  +\varepsilon\sum _{\bnu\in\Ind(\lambda) }\bnu_{\Box}\left(\sum_{i=1}^{\bn_{\Box}} \alpha_{\bnu,i}   x^{\bn_{\Box}-i}  \right)\right)_{\Box\in \lambda} = {}\\
& {} = \left(x^{\bn_{\Box} }  +\varepsilon \sum_{i=1}^{\bn_{\Box}} \left(\sum_{\bnu\in\Ind(\lambda) }\bnu_{\Box}\alpha_{\bnu,i} \right)  x^{\bn_{\Box}-i}  \right)_{\Box\in \lambda}
    \end{split}
\]
   with the assumption $\alpha_{i,\bnu}=0$, for $i>n_{\bnu}$. Assume now that the differential $\mathrm{d}\varphi_V$ is not injective at $[Y]$ and let $v$ be a nonzero vector such that $\mathrm{d}\varphi_{V, [Y]}(v)=0$. This implies that there exists a certain index $i$ such that
\[
    \sum _{\bnu\in\Ind(\lambda) } \bnu_\Box\alpha_{\bnu,i}=0,
\]  
      for all $\Box\in\lambda$, and such that not all $\alpha_{\bnu, i}$ are nonzero. Notice that up to scaling, such coefficients can be taken to be integers.
    Ranging over all boxes $\Box \in \lambda$ yields the required nontrivial relation
      \[\sum _{\bnu\in\Ind(\lambda) }\bnu  \alpha_{\bnu,i} =0, \]
    which clearly involves only the indicators $\bnu$ appearing in the factorisation $T$.
      The converse is analogous. In fact, assume there is a nontrivial relation 
      \[
        \sum_{\bnu\in \Ind(\lambda)}m_{\bnu}\cdot\bnu=0,
      \]
      involving only the indicators $\bnu$ appearing in the factorisation $T$. Then the nonzero vector 
        \[
    v=\left( x^{n_{\bnu}} +  \varepsilon   m_{\bnu}  x^{n_{\bnu}-1}\right)_{\bnu\in\Ind(\lambda)}\in T_{[Y]} \prod_{\bnu \in\Ind(\lambda)}C^{(n_{\bnu})}
    \]
    belongs to the kernels of $\mathrm{d}\varphi_{V, [Y]}$.
\end{proof}

\begin{remark}\label{rmk:inj-diff-implies-bijection}
Combining \Cref{prop: bijec comb} and \Cref{prop: inj diff comb} with one another shows that the injectivity of $\mathrm{d}\varphi_V$ implies the bijectivity of $\varphi_V$ on points. The latter condition can thus be dropped in \Cref{thm:smoothness-components}\,\ref{equiv-cond-4}.
\end{remark}

We can now state  the combinatorial criterion for smoothness, which will conclude the proof of \Cref{mainB} from the Introduction.

\begin{theorem}\label{thm: main thm for criterion of smooth}
Let $V\subset C^{[\bn]}$ be an irreducible component, corresponding to a factorisation $T=(n_{\bnu})_{\bnu \in \Ind(\lambda)}$. Then $V$ is smooth if and only if there is no nontrivial relation
\begin{equation}\label{eq:singrelation2new}
        \sum_{\bnu\in \Ind(\lambda)}m_{\bnu}\cdot\bnu=0,
    \end{equation}
where $m_{\bnu}\in \BZ$ are integers and are possibly nonzero only for the indicators $\bnu$ appearing in the factorisation $T$.
\end{theorem}
\begin{proof}
This follows combining \Cref{thm:smoothness-components}, \Cref{prop: bijec comb}, \Cref{prop: inj diff comb} and \Cref{rmk:inj-diff-implies-bijection} with one another.
\end{proof}

\Cref{ex:big} illustrates why the combinatorial criteria of Propositions \ref{prop: bijec comb}, \ref{prop: inj diff comb} are not equivalent.

\begin{example}\label{ex:big}
Consider the reverse plane partition 
\begin{center}
\begin{tikzpicture}
  \node at (-1.3,0) {$\bn=$};
        \node at (0,0) {\young(003,025,355)};
\end{tikzpicture}
\end{center}
and let $V\subset C^{[\bn]}$ be the irreducible component associated to the factorisation 
  \begin{center}
  \begin{tikzpicture}
  \node at (-1.3,0) {$\bn=$};
  \node at (0,-1.) {$\bnu_1$};
\node at (0,0) { \young(001,001,111)};
\node at (1,0) {$+$};
  \node at (2,-1.) {$\bnu_2$};
\node at (2,0) {\young(000,011,011)};
\node at (3,0) {$+$};
  \node at (4,-1.) {$\bnu_3$};
\node at (4,0) {\young(000,001,111)};
\node at (5,0) {$+$};
  \node at (6,-1.) {$\bnu_4$};
\node at (6,0) {\young(001,001,011)};
\node at (7,0) {$+$};
  \node at (8,-1.) {$\bnu_5$};
\node at (8,0) {\young(001,011,111)}; 
  \end{tikzpicture}
  \end{center}
  By a simple computation, one can see that all the relations among $\bnu_1, \dots, \bnu_5$ are generated by 
  \[
  2\bnu_1+\bnu_2=\bnu_3+\bnu_4+\bnu_5,
  \]
  which implies that $V$ is singular, $\varphi_V$ does not have injective differential but it is a bijection on points.
\end{example}
 
\begin{corollary}\label{cor:smoothle3}
Let $\bn$ be a reverse plane partition and $C$ a smooth curve. If $\omega(\bn)\le 3$, all the irreducible components of $C^{[\bn]}$ are smooth.
\end{corollary}
\begin{proof}
    By definition, since $\omega(\bn)\leq 3$, then any factorisation of $\bn$ has length at most 3. If $\bn=3\bnu_1 $ or $\bn=2\bnu_1+\bnu_2$, with $\bnu_1\neq \bnu_2$, then a relation of the form \eqref{eq:singrelation2new} clearly cannot exist. Assume therefore  that the factorisation has the form $\bn=\bnu_1+\bnu_2+\bnu_3$, with $\bnu_i$ three different indicators for $i=1,2,3$. 
 If the corresponding component is singular, by Theorem \ref{thm: main thm for criterion of smooth} there holds a relation of the form
 \begin{align}\label{eqn: boooh}
      m_1\bnu_1=m_2\bnu_2+m_3\bnu_3,
 \end{align}
such that $m_{i}>0$ for $i=1,2,3$. Denote by $U_i$ the supporting upper sets of $\bnu_i$, i.e. the collection of boxes of $\lambda$ such that $\bnu_i(\Box)\neq 0$, for all $i=1, 2,3$.
We have that $U_1= U_2\cup U_3 $. For the first implication,  assume we have a box $\Box\in U_1$ but  $\Box\notin U_2, U_3$. Then \eqref{eqn: boooh} evaluated at $\Box$ yields $m_1=0$, which is a contradiction. Conversely, assume we have a box $\Box \in U_2\cup U_3$ but $\Box\notin U_1$. Then  \eqref{eqn: boooh} evaluated at $\Box$ yields $m_2\bnu_2(\Box)+m_3\bnu_3(\Box)=0$, a contradiction. Therefore, there is a decomposition 
\[
U_1=(U_2\setminus U_3)\, \sqcup (U_3\setminus U_2)\, \sqcup (U_2\cap U_3).
\]
Since $U_1$ is connected, we have that  $U_2\cap U_3\neq \varnothing$, which by \eqref{eqn: boooh} implies that $m_1=m_2+m_3$. Finally, since $U_2\neq U_3$, we can assume with loss of generality that there exists a box $\Box\in U_2\setminus U_3$. Evaluating \eqref{eqn: boooh} at $\Box$ implies that $m_1=m_2$, from which we deduce that $m_3=0$, a contradiction.
\end{proof}

\begin{remark}
 The bound in \Cref{cor:smoothle3} is optimal. Indeed, \Cref{ex:singular} offers an example where an irreducible component is singular and corresponds to a factorisation of length 4.   
\end{remark}

\smallbreak
It is worth mentioning, at this point, that there are two distinguished classes of components of the double nested Hilbert scheme.

\begin{definition}\label{def: standard component}
Given a smooth curve $C$, we define  the \emph{standard component}, denoted by  $C_{\sta}^{[\bn]}$, as the irreducible component of $C^{[\bn]}$ associated to the standard factorisation of $\bn \in \BRPP^\lambda$, 
(cf. Definition \ref{def: standard}). Moreover, if $\bn$ admits a complete factorisation, we call its associated irreducible component \emph{complete}.
\end{definition}
 
These distinguished components are always smooth.

\begin{prop}\label{prop:standardsmooth}
  Let $\bn$ be a reverse plane partition and  $C$ be a smooth curve. Then, the standard component $C^{[\bn]}_{\sta}$ is smooth.
\end{prop}
\begin{proof}
Consider the standard factorisation $T$ given by
\[
\bn = \sum_{s=1}^t\sum_{i=1}^{r_s} \,( k_s-k_{s-1})\chi_{U_{s,i}},
\]
with notation as in \Cref{def: standard}. By  \Cref{thm: main thm for criterion of smooth},  we just need to prove that we cannot write relations of the form \eqref{eq:singrelation2new} among the indicators $\chi_{U_{s,i}}$.

Assume that there is such a relation 
\begin{align}\label{eqn: standard proof}
    \sum_{s=1}^t\sum_{i=1}^{r_s} m_{s,i}\chi_{U_{s,i}}=0.
\end{align}
For $i=1, \dots, r_1$,  take a box $\Box\in U_{1,i}\setminus \left(\coprod_{j=1}^{r_2}U_{2,j} \right)$. Then, \eqref{eqn: standard proof} implies that
\[
 m_{1,i}\chi_{U_{1,i}}(\Box)=0,
\]
since $\chi_{U_{s,l}}(\Box)=\chi_{U_{1,j}}(\Box)=0$ for all $ s\geq 2$, all $l=1, \dots, r_s $ and all $j\neq i$. Therefore, $m_{1,i}=0$ for all $i=1, \dots, r_1$. Next, consider a box $\Box\in U_{2,i}\setminus \left(\coprod_{j=1}^{r_3}U_{3,j} \right)$. By the same reasoning as above, we have that $m_{2,i}=0$ for $i=1, \dots, r_2 $. Iterating this process yields that $m_{s,i}=0 $ for all $s=1, \dots, t$ and all $i=1, \dots, r_s$, which contradicts nontriviality of the relation and concludes the proof.
\end{proof}
\begin{remark}
   The very same proof of \Cref{prop:standardsmooth} holds by replacing the standard component with  any irreducible component whose Young indicators appearing in the corresponding  factorisation are totally ordered,  with respect to the natural partial order of $\BRPP^\lambda$.  
\end{remark}

\begin{prop}\label{prop:completesmooth}
  Let $\bn$ be a reverse plane partition admitting a complete factorisation and  $C$ be a smooth curve. Then, the complete irreducible component  of $C^{[\bn]}$ is smooth.
\end{prop}
\begin{proof}
Let $T=(n_{\bnu})_{\bnu \in \Ind(\lambda)}$ be the complete factorisation. By \Cref{thm: main thm for criterion of smooth},  we just need to prove there are no relations of the form \eqref{eq:singrelation2new}. Assume there is such a relation
\begin{align}\label{eqn: complete prooof}
    \sum_{\underset{\mbox{\tiny $\bnu \in \Ind(\lambda)$}}{\mbox{\tiny $m_{\bnu}\ge 0$}}} {m_{\bnu}} \cdot\bnu =
\sum_{\underset{\mbox{\tiny $\bnu \in \Ind(\lambda)$}}{\mbox{\tiny $m_{\bnu}< 0$}}} -{m_{\bnu}}\cdot\bnu.
\end{align}
Both left-hand-side and right-hand-side of \eqref{eqn: complete prooof} are complete factorisations of the same reverse plane partition, which implies that $m_{\bnu}=0$ for all indicators $\bnu$ by the unicity of complete factorisations, see \cref{cor:uniquecomplete}.
\end{proof}

\section{Virtual structure, reducedness and local equations}\label{sec:virtual}

\subsection{Zero locus description and reducedness}\label{subsec:zero-locus-description}
We recall from \cite[Sec.~2.4]{Mon_double_nested} the description of the double nested Hilbert scheme of points on a smooth curve as the zero locus of a section of a vector bundle inside a smooth ambient space.

Let $C$ be a smooth curve, $ \bn$ a reverse plane partition with shape $\lambda$ and  define 
\begin{align*}
A_{C,\bn}=C^{(\bn_{0,0})}\times \prod_{\substack{(i,j)\in \lambda\\ i\geq 1}} C^{(\bn_{i,j}-\bn_{i-1,j})}\times \prod_{\substack{(l,k)\in \lambda\\ k\geq 1}} C^{(\bn_{l,k}-\bn_{l,k-1})}.
\end{align*}
The ambient space $A_{C,\bn}$ is a smooth quasiprojective variety of dimension 
\[
\dim A_{C,\bn}=\bn_{0,0}+\sum_{\substack{(i,j)\in \lambda\\ i\geq 1}}(\bn_{i,j}-\bn_{i-1,j})+\sum_{\substack{(l,k)\in \lambda\\ k\geq 1}}(\bn_{l,k}-\bn_{l,k-1}).
\]
To ease the notation, we denote its elements  by 
\begin{equation*}\label{eqn:point-acn}
\underline{Z}=((Z_{0,0}, X_{i,j},Y_{l,k}))_{(i,j),(l,k)}\in A_{C,\bn}    
\end{equation*} 
where
$Z_{0,0}\subset C$ is a divisor of length $\bn_{0,0}$ and $X_{i,j}\subset C$ (resp.~$Y_{l,k}\subset C$) is a divisor  of length $ \bn_{i,j}-\bn_{i-1,j}$ (resp.~$\bn_{l,k}-\bn_{l,k-1}$). By \cite[Thm.~2.8]{Mon_double_nested} there exists a locally free sheaf $\CE_{C, \bn}$ on $A_{C, \bn}$ of rank $ \sum_{\substack{(i,j)\in \lambda\\ i,j\geq 1}}(\bn_{i,j}-\bn_{i-1,j-1}) $  whose fibres satisfy
\[
\CE_{C, \bn}|_{\underline{Z}}=\bigoplus_{\substack{(i,j)\in \lambda\\ i,j\geq 1}}H^0(C, \OO_{Y_{i,j}+X_{i,j-1}}(X_{i,j}+Y_{i-1,j})).
\]

\begin{theorem}[{\cite[Thm.~2.8]{Mon_double_nested}}]\label{thm: zero locus}
Let $C$ be a smooth curve and $\bn$ be a reverse plane partition. Then, there exists a section $s$ of $\CE_{C,\bn}$ whose zero locus is isomorphic to $C^{[ \bn]}$.
\[
\begin{tikzcd}[row sep=large,column sep=large]
    &\CE_{C,\bn}  \arrow[d] \\
   C^{[ \bn]}\cong Z(s)\arrow[r, hook]
&A_{C,\bn} \arrow[bend right, swap]{u}{s}
\end{tikzcd}
\]
 \end{theorem}
We prove now that this section\footnote{We do not need the actual definition of the section $s$ for the proof of this result, but for its general definition  see {\cite[Theorem 2.8]{Mon_double_nested}}.} cuts out $C^{[ \bn]}$ \emph{transversally} (i.e., in the expected codimension) inside the smooth ambient locus $ A_{C,\bn}$. 

\begin{theorem}\label{thm: complete interse}
Let $C$ be an irreducible smooth quasiprojective curve and let $\bn$ be a reverse plane partition with shape $\lambda$, then
\[
\dim  A_{C,\bn} - \rk \CE_{C,\bn}= \omega(\bn).
\]
In particular,  the immersion in  \Cref{thm: zero locus} is regular and $  C^{[\bn]}$ is a local complete intersection and reduced.
\end{theorem}

\begin{proof}
We prove the first claim by induction on the size of $\lambda$. If $\lvert \lambda\rvert=1$, then $\rk \CE_{C,\bn}=0$ and the claim trivially holds. Assume now that the claim holds for a partition  $\lambda$  and consider the partition $\widetilde{\lambda}$ obtained by $\lambda$ by adding a box $\Box\in\BN^2$. Then, any reverse plane partition $\widetilde{\bn}\in\BRPP^{\widetilde{\lambda}} $ is obtained by a reverse plane partition $\bn\in\BRPP^\lambda $ by labelling the extra box $\Box\in\widetilde{\lambda}\smallsetminus \lambda$ by $\widetilde{\bn}_\Box$. \\
\emph{Case I.} If $\Box = (i,0)$ (the case $(0,j)$ is analogous), then 
\begin{align*}
     A_{C,\widetilde{\bn}} &\cong  A_{C,\bn}  \times C^{(\bn_{i,0}-\bn_{i-1,0})},\\
     \rk  \CE_{C,\widetilde{\bn}}&= \rk \CE_{C,\bn},
\end{align*}
which implies, by the inductive hypothesis, that
\begin{align*}
  \dim  A_{C,\widetilde{\bn}} - \rk \CE_{C,\widetilde{\bn}}&=\omega(\bn) + \bn_{i,0}-\bn_{i-1,0}\\
  &= \omega(\widetilde{\bn}).
\end{align*}
\emph{Case II.} If $\Box = (i,j)$ with $i,j\geq 1$, then
\begin{align*}
     A_{C,\widetilde{\bn}} &\cong  A_{C,\bn}  \times C^{(\bn_{i,j}-\bn_{i-1,j})}\times C^{(\bn_{i,j}-\bn_{i,j-1})},\\
     \rk  \CE_{C,\widetilde{\bn}}&= \rk \CE_{C,\bn}+\bn_{i,j}-\bn_{i-1, j-1},
\end{align*}
which implies by the inductive hypothesis
\begin{align*}
  \dim  A_{C,\widetilde{\bn}} - \rk \CE_{C,\widetilde{\bn}}&=\omega(\bn) + \bn_{i,j}-\bn_{i-1,j} +\bn_{i,j}-\bn_{i,j-1}-\bn_{i,j}+\bn_{i-1, j-1}\\
  &= \omega({\widetilde{\bn}}).
\end{align*}
The fact that $\dim  A_{C,\bn} - \rk \CE_{C,\bn}= \omega(\bn)$, combined with \Cref{thm:irr components},
 implies that the codimension  of the immersion $C^{[ \bn]}\hookrightarrow A_{C,\bn}$ agrees with $\rk \CE_{C,\bn}$, therefore, the immersion is regular and $C^{[\bn]} $ is a local complete intersection.  In particular, it is Cohen--Macaulay and  it satisfies Serre's property ($\mathrm{S}_1$). On the other hand, $C^{[\boldit{n}]}$ satisfies Serre's property ($\mathrm{R}_0$), i.e., it is regular in codimension $0$, by \Cref{prop:conn-comp-free}. Therefore, it is reduced by \cite[\href{https://stacks.math.columbia.edu/tag/033P}{Tag 033P}]{stacks-project}.
\end{proof}



\subsection{Virtual fundamental class}\label{sec: virtual fundamental}
A \emph{perfect obstruction theory} \cite{BF_normal_cone} on a scheme $X$ is the datum of a morphism 
\[
\begin{tikzcd}
\phi\colon \BE\arrow{r} & \BL_X
\end{tikzcd}
\]
in the derived category $\derived^{[-1,0]}(X)$, where $\BE$ is a perfect complex of perfect amplitude contained in $[-1,0]$, such that $h^0(\phi)$ is an isomorphism and $h^{-1}(\phi)$ is surjective. Here, $\BL_X = \tau_{\geq -1}L_X^\bullet$ is the cut-off at $-1$ of the full cotangent complex $L_X^\bullet \in \derived^{\leq 0}(X)$.
The \emph{virtual dimension} of $X$ with respect to $(\BE,\phi)$ is the integer $\vd  = \rk \BE$. This is just $\rk E^0 - \rk E^{-1}$ if one can write $\BE = [E^{-1}\to E^0]$, for two locally free sheaves $E^{-1}$ and $E^0$. A perfect obstruction theory determines a cone
\[
\begin{tikzcd}
\mathfrak C \arrow[hook]{r} &  E_1 = (E^{-1})^*.
\end{tikzcd}
\]
Letting $\iota\colon X \into E_1$ be the zero section of the vector bundle $E_1$, the induced \emph{virtual fundamental class} on $X$ is the refined intersection
\[
[X]^{\vir} = \iota^![\mathfrak C]\, \in\, A_{\vd}(X). 
\]
Thanks to \Cref{thm: zero locus}, $ C^{[ \bn]}$ inherits a perfect obstruction theory
\begin{equation*}\label{symmetric_POT_quot}
\begin{tikzpicture}[baseline=(current  bounding  box.center)]
\node (N1) at (-1.98,0.95) {$\mathbb E$};
\node (N2) at (-1.1,0.94) {$=$};
\node (N3) at (-1.97,-0.88) {$\BL_{C^{[ \bn]}}$};
\node (N4) at (-1.1,-0.88) {$=$};
\node (O1) at (0.1,0.93) {$\big[\CE_{C,\bn}^*\big|_{C^{[ \bn]}}$};
\node (O2) at (2.99,0.93) {$\Omega_{A_{C,\bn}}\big|_{C^{[ \bn]}}\big]$};
\node (O3) at (0.1,-0.88) {$\big[ \mathscr{I}/\mathscr{I}^2$};
\node (O4) at (2.99,-0.88) {$\Omega_{A_{C,\bn}}\big|_{C^{[ \bn]}}\big]$};
\path[commutative diagrams/.cd, every arrow, every label]
(N1) edge node[swap] {$\phi$} (N3)
(O1) edge node {$\dd s^*$} (O2)
(O1) edge node[swap] {$s^*$} (O3)
(O3) edge node {$\dd$} (O4)
(O2) edge node {$\mathrm{id}$} (O4);
\end{tikzpicture}
\end{equation*}
where $\mathscr{I}$ is the ideal sheaf of the immersion $C^{[ \bn]}\hookrightarrow A_{C,\bn}$, 
and, therefore, a virtual fundamental class $[C^{[ \bn]}]^{\vir}\in A_{\vd}(C^{[ \bn]})$, where, in this case, one has $\vd = \dim A_{C,\bn} - \rk \CE_{C,\bn}$.

\begin{corollary}\label{cor:VFC=FC}
There is an identity 
\[
\bigl[C^{[ \bn]}\bigr]^{\vir}=\bigl[C^{[\bn]}\bigr]\in A_{\omega(\bn)}(C^{[ \bn]}).
\]
\end{corollary}
\begin{proof}
Since the immersion $C^{[ \bn]}\hookrightarrow A_{C,\bn}$ is regular, we have that $\mathscr{I}/\mathscr{I}^2 $ is a locally free sheaf of rank $\rk \CE_{C,\bn} $ (see e.g. \cite[Lemma 6.3.6]{Liu_AG}), which readily implies that the surjection 
\[
\begin{tikzcd}
\CE_{C,\bn}^*\big|_{C^{[\bn]}} \arrow[two heads]{r}{s^\ast} & \mathscr{I}/\mathscr{I}^2
\end{tikzcd}
\]
is an isomorphism. Therefore, the perfect obstruction theory on $C^{[\bn]}$ is an isomorphism $\phi\colon\BE\simto\BL_{C^{[ \bn]}}$, which implies that the virtual fundamental class coincides with the naive fundamental class.
\end{proof}

The statement of \Cref{cor:VFC=FC} implies the following. If $C$ is projective of genus $g$, then, for any cohomology class $\alpha \in \HH^\ast(C^{[\bn]},\BQ)$, there is an identity
\[
\int_{[C^{[\bn]}]^{\vir}} \alpha = \int_{C^{[\bn]}}\alpha\,\in\,\BQ.
\]
Moreover, as in \cite[Sec.~5.3]{Mon_double_nested}, choosing suitably large $N_i\gg 0$, there are morphisms
\[C^{[\bn]}\hookrightarrow \prod_{i}C^{(N_i)}\to \prod_{i}\mathrm{Pic}^{N_i}(C), \]
where the first map is a regular embedding (by \Cref{thm: complete interse}) and the second one is a product of Abel--Jacobi maps, each of which is a projective bundle. Informally, this means that virtual intersection theory on $C^{[\bn]}$ agrees with standard intersection theory, and that it can be expressed as intersection theory on Jacobians of $C$, which are all diffeomorphic to a  $g$-dimensional torus (when $\bfk=\BC$).

\subsection{Local equations}\label{sec:local-eqns}
 In this subsection, we work \'etale-locally and we assume $C=\mathbb{A}^1$.
 We look for local equations of the double nested Hilbert scheme $(\mathbb{A}^1)^{[\bn]}$. We provide equations of two different closed immersions of the double nested Hilbert scheme $(\mathbb{A}^1)^{[\bn]}$ into (different) smooth quasiprojective varieties. Precisely, the first is the natural immersion in the product of the Hilbert schemes $(\BA^1)^{(\bn_{i,j})}$ for $(i,j)\in\lambda$, and the second is inspired by \Cref{subsec:zero-locus-description}.

\medskip

\paragraph{\it Type I} The first type of equations describes the nested Hilbert scheme $(\mathbb{A}^1)^{[\bn]}$ as a closed subscheme of the product of classical Hilbert schemes $\prod_{(i,j) \in \lambda} (\mathbb{A}^1)^{(\bn_{i,j})}$. Each factor $(\mathbb{A}^1)^{(\bn_{i,j})}$ is isomorphic to the affine space $\mathbb{A}^{\bn_{i,j}} = \Spec \mathbf{k}[a_{i,j,k}]_{k=1}^{\bn_{i,j}}$ where the coordinates $a_{i,j,k}$ identify the coefficients of the monic polynomial of degree $\bn_{i,j}$ 
\begin{equation}\label{eq: polinomi tipo vertex}
P_{i,j}(x) = x^{\bn_{i,j}} + \sum_{k=1}^{\bn_{i,j}} a_{i,j,k}\, x^{\bn_{i,j} - k}
\end{equation}
defining the universal family in $(\mathbb{A}^1)^{(\bn_{i,j})} \times \mathbb{A}^1$. The double nested Hilbert scheme $(\mathbb{A}^1)^{[\bn]}$ can be realised as a subscheme of the affine space 
\[
\prod_{(i,j) \in \lambda} \mathbb{A}^{\bn_{i,j}} \simeq \mathbb{A}^{\lvert \bn \rvert} = \Spec \mathbf{k}[a_{i,j,k}]_{\begin{subarray}{l}
    (i,j) \in \lambda\\1\le k\le \bn_{i,j}
\end{subarray}}
\] 
by imposing the divisibility between polynomials corresponding to adjacent boxes. Namely, for all $(i,j)\in \lambda$,
\begin{equation}\label{eq: type i}
    \begin{split}  
    &P_{i,j}(x) \equiv 0 \mod P_{i-1,j}(x),\quad \text{if~}i>0,\\ 
    &P_{i,j}(x) \equiv 0 \mod P_{i,j-1}(x), \quad \text{if~}j>0.
    \end{split}
\end{equation}

Notice that, for each box $(i,j) \in \lambda$, we impose $\bn_{i-1,j} + \bn_{i,j-1} - \bn_{i-1,j-1}$ conditions on the polynomial $P_{i,j}(x)$. This is a consequence of the inclusion-exclusion principle. Indeed, if $i>0$ and $j>0$, the roots of $P_{i-1,j}(x)$ and $P_{i,j-1}(x)$ are roots of $P_{i,j}(x)$ as well, but the roots of $P_{i-1,j-1}(x)$ are counted twice (if $i=0$ or $j=0$, the formula is still true recalling \cref{con:zero-outside}). This is a different way to show that
\begin{align*}
\dim (\mathbb{A}^1)^{[\bn]} &\geq 
 \dim \mathbb{A}^{\lvert \bn \rvert} - \sum_{(i,j)\in \lambda} \left(\bn_{i-1,j} + \bn_{i,j-1} - \bn_{i-1,j-1}\right) \\
&= \sum_{(i,j)\in \lambda} \left(\bn_{i,j} - \bn_{i-1,j} - \bn_{i,j-1} + \bn_{i-1,j-1}\right) = \lvert \Delta\bn \rvert = \omega(\bn),
\end{align*}
as already done in the proof of \cref{thm: complete interse}.

\medskip

\paragraph{\it Type II} Instead of giving the universal family of configurations, one can consider the collection of families parametrising differences of subschemes corresponding to adjacent boxes in a configuration. This viewpoint leads to the realisation of the double nested Hilbert scheme $(\mathbb{A}^1)^{[\bn]}$ as a closed subscheme of the following product of classical Hilbert schemes
\[
\prod_{(i,j)\in\lambda} \left( (\mathbb{A}^1)^{(\bn_{i,j}-\bn_{i-1,j})} \times (\mathbb{A}^1)^{(\bn_{i,j}-\bn_{i,j-1})}\right).
\]
In fact, this interpretation is (almost) the one given in \Cref{subsec:zero-locus-description}. 

Let us denote by $b_{i,j,k}$, for $k = 1,\ldots,\bn_{i,j}-\bn_{i-1,j}$, and $c_{i,j,k}$, for $k = 1,\ldots,\bn_{i,j}-\bn_{i,j-1}$, the coordinates of $(\mathbb{A}^1)^{(\bn_{i,j}-\bn_{i-1,j})} \times (\mathbb{A}^1)^{(\bn_{i,j}-\bn_{i,j-1})} \cong \mathbb{A}^{\bn_{i,j}-\bn_{i-1,j}}\times \mathbb{A}^{\bn_{i,j}-\bn_{i,j-1}}$. Then, for all $(i,j) \in \lambda$, the following polynomials
\begin{equation}\label{eq: polinomi tipo edge}
\begin{split}
L_{i,j}(x) &{}= x^{\bn_{i,j}-\bn_{i-1,j}} + \sum_{k=1}^{\bn_{i,j}-\bn_{i-1,j}} b_{i,j,k}\, x^{\bn_{i,j}-\bn_{i-1,j} - k},\\
U_{i,j}(x) &{}= x^{\bn_{i,j}-\bn_{i,j-1}} + \sum_{k=1}^{\bn_{i,j}-\bn_{i,j-1}} c_{i,j,k}\, x^{\bn_{i,j}-\bn_{i,j-1} - k},
\end{split}
\end{equation}
are the equations of the universal families of $(\mathbb{A}^1)^{(\bn_{i,j}-\bn_{i-1,j})} $ and $ (\mathbb{A}^1)^{(\bn_{i,j}-\bn_{i,j-1})}$ respectively. We interpret them as the families $\mathcal{Z}_{i,j} - \mathcal{Z}_{i,j-1}$ and $\mathcal{Z}_{i,j} - \mathcal{Z}_{i-1,j}$ of differences of universal divisors corresponding to adjacent boxes of a configuration with reverse plane partition $\bn$. Now, to obtain the equations defining $(\mathbb{A}^1)^{[\bn]}$, we have to impose, for all $(i,j)\in\lambda$, the condition coming from the nesting, i.e., that the difference $\mathcal{Z}_{i,j} - \mathcal{Z}_{i-1,j-1}$ is given by
\[
(\mathcal{Z}_{i,j}- \mathcal{Z}_{i,j-1} ) + (\mathcal{Z}_{i,j-1} - \mathcal{Z}_{i-1,j-1})= (\mathcal{Z}_{i,j}- \mathcal{Z}_{i-1,j} ) + (\mathcal{Z}_{i-1,j} - \mathcal{Z}_{i-1,j-1}).
\]
This leads, for all $(i,j)\in\lambda$, to the equations
\begin{equation}\label{eq: type ii}
    L_{i,j}(x) U_{i-1,j}(x) = U_{i,j}(x) L_{i,j-1}(x),
\end{equation}
with the convention that $U_{-1,j}(x) = L_{i,-1}(x) = 1$.
These polynomials have degree $\bn_{i,j} - \bn_{i-1,j-1}$, thus, we are imposing $\bn_{i,j} - \bn_{i-1,j-1}$ conditions on the coordinates $b_{i,j,k}$ and $c_{i,j,k}$.
Again, one has
\begin{align*}
\dim (\mathbb{A}^1)^{[\bn]} 
&\geq 
\sum_{(i,j)\in \lambda} \left(2\bn_{i,j}-\bn_{i-1,j}-\bn_{i,j-1}\right) - \sum_{(i,j)\in \lambda} \left(\bn_{i,j} - \bn_{i-1,j-1}\right)\\
&= \sum_{(i,j)\in \lambda} \left(\bn_{i,j}-\bn_{i-1,j}-\bn_{i,j-1} + \bn_{i-1,j-1}\right) = \lvert \Delta\bn \rvert = \omega(\bn).
\end{align*}

From a computational viewpoint, 
this approach can be made more efficient, since the  polynomials $L_{i,0}(x)$ and $U_{0,j}(x)$ are not really necessary (cf.~\cref{subsec:zero-locus-description}). Hence, $(\mathbb{A}^1)^{[\bn]}$ may be embedded in a smaller affine scheme.

\begin{remark}\label{rk:equations-degree}
In both constructions, the ideal defining the double nested Hilbert scheme is homogeneous with respect to the nonstandard positive grading
\begin{equation}\label{eq:degree}
\begin{array}{rl}
\textit{(Type I)} & \deg a_{i,j,k} = k\\
\textit{(Type II)} &\deg b_{i,j,k} = \deg c_{i,j,k} = k
\end{array}\qquad \forall\ (i,j) \in \lambda.
\end{equation}
Indeed, the equations defining $(\mathbb{A}^1)^{[\bn]}$  come from algebraic manipulation of the polynomials in \eqref{eq: polinomi tipo vertex} and \eqref{eq: polinomi tipo edge} which are all homogeneous with respect to $\deg x = 1$ and the degrees in \eqref{eq:degree}. Therefore, $(\mathbb{A}^1)^{[\bn]}$ has a  $\mathbf{k}^\star$-action, defined on the coordinates of the ambient spaces by
\[
\begin{array}{rl}
\textit{(Type I)} & \ \ t \curvearrowleft a_{i,j,k} = t^k a_{i,j,k}\\
\textit{(Type II)} & \begin{array}{l} t\curvearrowleft b_{i,j,k} = t^k b_{i,j,k},\\[-3pt]  t\curvearrowleft c_{i,j,k} = t^k c_{i,j,k}\end{array}
\end{array}
 \qquad \forall\ (i,j) \in \lambda,\ \forall\ t \in \mathbf{k}^\star.
\]
See \cite[Thm.~3.2]{FerrareseRoggero} for a general description of the properties of this type of varieties and \cite[Sec.~4]{LellaRoggero} for an application to classical Hilbert schemes. 
One of the main feature of this kind of varieties is that they can be embedded in the tangent space at the origin. This implies that the dimension of the tangent space of $(\mathbb{A}^1)^{[\bn]}$ at the origin is the maximum possible. In particular, this gives the smallest affine space in which $(\mathbb{A}^1)^{[\bn]}$ can be embedded.
\end{remark}

\section{The motive of the double nested Hilbert scheme}
In this section, we 
prove a general structural result (\Cref{thm:power-structure-doublenested}), which will easily imply \Cref{thm:main-C} from the introduction. In order to achieve this, we recall the (multiple variable) power structure over $K_0(\Var_{\bfk})$ and its `geometric description' following \cite{BM15,GLMps,GLMps2}. We refer the reader to \cite{DL_geometry_arc_spaces} for background on the Grothendieck ring of varieties.

\begin{notation}
Just for this section, we adopt a different notation for partitions, in order to be consistent with the literature on power structures. Let $n$ be a positive integer. We write a partition $\alpha \vdash n$ as a tuple $\alpha = (1^{\alpha_1}2^{\alpha_2}\cdots n^{\alpha_n})$, meaning that $\alpha$ has $\alpha_i$ parts of length $i$ for every $i$ (ranging from $1$ to $n$). Of course $\alpha_i$ may be $0$ for several values of $i$. In this language, the size of $\alpha$ is $n = \lvert \alpha \rvert = \sum_ii\alpha_i$. We also set $|| \alpha || = \sum_{i} \alpha_i$, representing the number of parts of $\alpha$. We write $G_\alpha$ for the automorphism group $ \prod_i\FS_{\alpha_i}$ of $\alpha$, where $\FS_b$ denotes the symmetric group on $b$ elements.
\end{notation}

\subsection{Power structures}\label{subsec:PS}
In this subsection, we recall the notion of power structure on a ring, following \cite{GLMps,GLMps2} closely. We refer the reader to \cite{DavisonR} for the $\lambda$-structure on $K_0(\Var_{\bfk})$, in more a general setting.


\begin{definition}[{\cite{GLMps}}]\label{def:power-structure}
A power structure on a ring $R$ is a map
\[
\begin{tikzcd}
(1+tR\llbracket t \rrbracket) \times R \arrow{r} &  1+tR\llbracket t \rrbracket,\quad (A(t),m) \mapsto A(t)^m,
\end{tikzcd}
\]
such that the following conditions are satisfied:
\begin{enumerate}
    \item [\mylabel{ps-i}{(i)}] $A(t)^1=A(t)$,
    \item [\mylabel{ps-ii}{(ii)}] $(A(t)B(t))^m = A(t)^m B(t)^m$,
    \item [\mylabel{ps-iii}{(ii)}] $A(t)^{m+m'} = A(t)^mA(t)^{m'}$,
    \item [\mylabel{ps-iv}{(iv)}] $A(t)^{mn} = (A(t)^{m})^n$,
    \item [\mylabel{ps-v}{(v)}] $(1+t)^m \equiv 1+mt\,(\bmod~t^2)$,
    \item [\mylabel{ps-vi}{(vi)}] $A(t^k)^m = A(t)^m|_{t \mapsto t^k}$.
\end{enumerate}
\end{definition}


\begin{example}\label{ex:ps}
We have the following examples of power structures.
\begin{itemize}
    \item [\mylabel{PS1}{(1)}] On $R=\BZ$, for any $m \in \BN$ we may set 
    \[
\left(1+\sum_{n>0}c_nt^n\right)^m
 = 1+\sum_{n>0}\left(\sum_{\alpha \vdash n} \prod_{k=0}^{|| \alpha ||-1} (m-k) \frac{\prod_{i} c_i^{\alpha_i}}{\prod_{i}\alpha_i!}\right)t^n.
    \]
    \item [\mylabel{PS2}{(2)}] On $R=\BZ[u,v]$, we may set 
    \[
    \left(1-t\right)^{-f} = \prod_{i,j} \,\left(1-u^iv^jt\right)^{-p_{ij}}
    \]
    if $f = \sum_{i,j} p_{ij} u^iv^j \in R$. This extends uniquely to a power structure \cite{GLMps2}. The interest in this ring lies in the fact that it is the value ring for the Hodge--Deligne polynomial of complex algebraic varieties.
    \item [\mylabel{PS3}{(3)}] Set $R=K_0(\Var_\bfk)$. If $(Y_n)_{n>0}$ is a sequence of varieties and $X$ is a variety, we set
    \[
\left(1+\sum_{n>0}\,[Y_n]t^n\right)^{[X]} = \sum_{n\geq 0}t^{n}\sum_{\alpha\vdash n}\,\left[\left(\prod_i X^{\alpha_i} \setminus \Delta\right) \times^{G_\alpha} \prod_i Y_i^{\alpha_i}\right],
    \]
where $G_{\alpha}$ acts diagonally, by permutation factor by factor, and $\Delta$ is the big diagonal, consisting of $||\alpha||$-tuples of points in $X$ such that at least two entries coincide.
This formula extends uniquely to a power structure by \cite[Thm.~2]{GLMps}.
\end{itemize}
\end{example}

\begin{remark}
For any variety $X$, we have
\[
(1-t)^{-[X]} = \sum_{n\geq 0}\,\bigl[X^{(n)}\bigr]t^n = \zeta_X(t),
\]
the Kapranov motivic zeta function of $X$, see \cite{Kapranov_rational_zeta}.
\end{remark}

\subsubsection{Multiple variable power structure on \texorpdfstring{$K_0(\Var_{\bfk})$}{K\_0(Var)}}\label{sec:multiple-ps}
We explain how to extend  \Cref{ex:ps}\,\ref{PS3} to the multiple variable setting, following \cite{GLMps2}. Let $\boldit{t} = (t_1,\ldots,t_e)$ be a finite set of indeterminates. We use the multi-index notation
\[
\boldit{t}^{\bn} = t_1^{n_1} \cdots t_e^{n_e}
\]
throughout, for $\bn = (n_1,\ldots,n_e)$ a tuple of integers $n_i \in \BZ_{\geq 0}$. Consider the power series ring $K_0(\Var_{\bfk}) \llbracket \boldit{t} \rrbracket = K_0(\Var_{\bfk}) \llbracket t_1,\ldots,t_e\rrbracket$, along with its ideal $\Fm = (t_1,\ldots,t_e) \subset K_0(\Var_{\bfk}) \llbracket \boldit{t} \rrbracket$. Thanks to the power structure of \Cref{ex:ps}\,\ref{PS3}, one can make sense of the powers $A(\boldit{t})^m$ for every $A(\boldit{t})\in 1+\Fm K_0(\Var_{\bfk}) \llbracket \boldit{t} \rrbracket$ and every $m \in K_0(\Var_{\bfk})$.  
As in the single variable case, there is an explicit formula in the effective case, which goes as follows. Given a sequence of varieties $Y_{\bn}$ indexed by elements $\bn \in \BZ_{\geq 0}^{e} \setminus \boldit{0}$, and a variety $X$, by \cite[Eqn.~(3)]{GLMps2} one has the identity
\begin{equation}\label{eqn:multiple-ps-coef}
\left(1 + \sum_{\bn \in \BZ_{\geq 0}^{e} \setminus \boldit{0}} \,[Y_{\bn}] \boldit{t}^{\bn}\right)^{[X]} = 1+\sum_{\bn \in \BZ_{\geq 0}^{e} \setminus \boldit{0}} \boldit{t}^{\bn} \sum_{\substack{\boldit{k}\in \CB \\ \sum_{\boldit{i}}\boldit{i} k_{\boldit{i}} = \bn}} \left[ \left(\prod_{\boldit{i} \,|\,k_{\boldit{i}} \neq 0} X^{k_{\boldit{i}}} \setminus \Delta \right) \times^{G_{\boldit{k}}}  \prod_{\boldit{i} \,|\,k_{\boldit{i}} \neq 0} Y_{\boldit{i}}^{k_{\boldit{i}}}\right]
\end{equation}
in $1+\Fm K_0(\Var_\bfk) \llbracket \boldit{t} \rrbracket$, where 
\[
\CB = \Set{(k_{\boldit{i}})_{\boldit{i} \in \BZ_{\geq 0}^{e} \setminus \boldit{0}} | k_{\boldit{i}} \in \BZ_{\geq 0},\,k_{\boldit{i}} = 0 \mbox{ for almost all }\boldit{i}} \subset \mathrm{Map}(\BZ_{\geq 0}^{e} \setminus \boldit{0},\BZ_{\geq 0})
\]
and each $\boldit{k} = (k_{\boldit{i}})_{\boldit{i} \in \BZ_{\geq 0}^{e} \setminus \boldit{0}} \in \CB$ determines an `automorphism group' 
\[
G_{\boldit{k}} = \prod_{\boldit{i} \,|\,k_{\boldit{i}} \neq 0} \FS_{k_{\boldit{i}}},
\]
acting diagonally on the two products in \eqref{eqn:multiple-ps-coef} just as in the single variable case.

\subsubsection{The geometric description of the multiple variable power structure on \texorpdfstring{$K_0(\Var_{\bfk})$}{K\_0(Var)}}\label{sec:geom-description}
Full details for this subsection can be found in the original papers \cite{GLMps,GLMps2} and in \cite[Sec.~2]{BM15}. 

Given a sequence of varieties $Y_{\bn}$ indexed by elements $\bn \in \BZ_{\geq 0}^{e} \setminus \boldit{0}$, we have a tautological map 
\[
\begin{tikzcd}
\BY = \displaystyle\coprod_{\bn\in \BZ_{\geq 0}^{e} \setminus \boldit{0}}Y_{\bn} \arrow{r}{\tau} & \BZ_{\geq 0}^{e} \setminus \boldit{0}
\end{tikzcd}
\]
collapsing the whole $Y_{\bn}$ onto the correct index, namely $\bn$. 

The `geometric description' of the power structure says the following: for every effective class $[X] \in K_0(\Var_{\bfk})$, the coefficient of $\boldit{t}^{\bn}$ in 
\[
\left( 1 + \sum_{\bn \in \BZ_{\geq 0}^{e} \setminus \boldit{0}} \,[Y_{\bn}] \boldit{t}^{\bn}\right)^{[X]},
\] 
which thanks to \eqref{eqn:multiple-ps-coef} is explicitly given by the class of the quotient variety
\[
\coprod_{\substack{\boldit{k}\in \CB \\ \sum_{\boldit{i}}\boldit{i} k_{\boldit{i}} = \bn}} 
\left(\prod_{\boldit{i} \,|\,k_{\boldit{i}} \neq 0} X^{k_{\boldit{i}}} \setminus \Delta \right) \times^{G_{\boldit{k}}}  \prod_{\boldit{i} \,|\,k_{\boldit{i}} \neq 0} Y_{\boldit{i}}^{k_{\boldit{i}}},
\]
agrees with the motivic class of the variety of pairs $(K,\phi)$ where $K \subset X$ is a finite subset and $\phi \colon K \to \BY$ is a function such that $\bn = \sum_{p \in K} \tau(\phi(p))$.

The geometric description will be used in the proof of \Cref{thm:power-structure-doublenested}.

\subsection{A general power structure formula}
Let $X$ be a smooth $\bfk$-variety of dimension $d$. Fix a Young diagram $\lambda$. We want to study the generating function
\[
\mathsf Z^\lambda_{X}(\boldit{q}) = \sum_{\bn \in \BRPP^\lambda} \,\bigl[X^{[\bn]}\bigr] \boldit{q}^{\bn} \in K_0(\Var_\bfk)\llbracket \boldit{q} \rrbracket,
\]
where the sum is over all reverse plane partitions with shape $\lambda$, and $\boldit{q}$ is a multivariable consisting of $\lvert \lambda \rvert$ variables $q_{\Box}$. We will exhibit an explicit formula in the case where $X=C$ is a curve.

There is a Hilbert-to-Chow type morphism \cite{Rydh1}
\[
\begin{tikzcd}
\rho_{X,\bn} \colon X^{[\bn]} \arrow{r} &  X^{(\lvert \bn \rvert)}
\end{tikzcd}
\]
sending the point $[Z] \in X^{[\bn]}$ corresponding to $Z=(Z_{\Box})_{\Box \in \lambda}$ to the $0$-cycle $\sum_{\Box \in \lambda} Z_{\Box}$.

Let $p \in X$ be a closed point. Define the \emph{punctual double nested Hilbert scheme} at $p$ to be the fibre
\[
X_p^{[\bn]} = \rho_{X,\bn}^{-1}(\lvert \bn \rvert p) \subset X^{[\bn]}.
\]
This locus parametrises tuples of subschemes all supported entirely at $p$. It does not depend on $p$, nor on $X$, but only on $d = \dim X$, by standard \'etale-local arguments exploiting the smoothness of $X$ (cf.~\cite[Lemma 1.3]{MR_nested_Quot}, \cite[Sec.~2.1]{Ric_motive_quot_locally_free} or \cite[Thm.~B]{Coh(X)-motivic}). In other words, there is a (noncanonical) isomorphism of schemes
\[
X_p ^{[\bn]} \cong (\BA^d)_0^{[\bn]},
\]
where $d = \dim X$, $p \in X$ is an arbitrary point and $0 \in \BA^d$ denotes the origin.
Define the punctual generating function
\[
\mathsf Z^\lambda_{\mathsf{punctual},d}(\boldit{q}) = \sum_{\bn \in \BRPP^\lambda}\,\bigl[(\BA^d)_0^{[\bn]}\bigr] \boldit{q}^{\bn} \in K_0(\Var_\bfk)\llbracket \boldit{q} \rrbracket.
\]

\begin{theorem}\label{thm:power-structure-doublenested}
Let $X$ be a smooth variety of dimension $d$. For every Young diagram $\lambda$, there is an identity
\[
\mathsf Z^\lambda_{X}(\boldit{q}) = \mathsf Z^\lambda_{\mathsf{punctual},d}(\boldit{q})^{[X]}.
\]
\end{theorem}

\begin{proof}
This follows at once from the geometric description of the power structure as recalled in \Cref{sec:geom-description}. We give the full argument for the sake of completeness; it is a simple adaptation of the proof of \cite[Thm.~1]{GLMps2} to the double nested case. 

By the discussion in \Cref{sec:geom-description},
the coefficient of $\boldit{q}^{\bn}$ in the right hand side, namely the class of the variety
\[
B_{\bn} = 
\coprod_{\substack{\boldit{k}\in \CB \\ \sum_{\boldit{i}}\boldit{i} k_{\boldit{i}} = \bn}} 
\left(\prod_{\boldit{i} \,|\,k_{\boldit{i}} \neq 0} X^{k_{\boldit{i}}} \setminus \Delta \right) \times^{G_{\boldit{k}}}  \prod_{\boldit{i} \,|\,k_{\boldit{i}} \neq 0} ((\BA^d)_0^{[\boldit{i}]})^{k_{\boldit{i}}},
\]
agrees with the variety of pairs
\[
(K,\phi)
\]
where $K \subset X$ is a finite subset and 
\[
\begin{tikzcd}
\phi \colon K \arrow{r} &  \displaystyle \coprod_{\boldit{k} \in \BZ_{\geq 0}^{\lvert \lambda \rvert} \setminus \boldit{0}}(\BA^d)_0^{[\boldit{k}]}
\end{tikzcd}
\]
is a map such that $\bn = \sum_{P \in K} \tau(\phi(P))$.
To conclude, it is enough to confirm that such a pair $(K,\phi)$ determines and is determined by a point $[Z] \in X^{[\bn]}$. 

Consider a point $[Z] \in X^{[\bn]}$, corresponding to the configuration $Z = (Z_{\Box})_{\Box \in \lambda}$. Let $P_1,\ldots,P_h \in X$ be the \emph{distinct} points appearing in $\Supp\left(\bigcup_{\Box \in \lambda} Z_{\Box}\right)$. For each $i=1,\ldots,h$ there is a reverse plane partition $\boldit{k}_i$ (of shape $\lambda$) determined by 
\[
\boldit{k}_{i,\Box} = \length \OO_{Z_{\Box}}|_{P_i},
\]
and an associated punctual double nested scheme 
\[
[Z_i] \in X_{P_i}^{[\boldit{k}_i]} \cong (\BA^d)_0^{[\boldit{k}_i]},
\]
corresponding to $Z_i = (Z_{i,\Box})_{\Box \in \lambda}$, satisfying $\chi(\OO_{Z_{i,\Box}}) = \boldit{k}_{i,\Box}$. It is constructed ignoring all $P_j$ with $j \neq i$ (note that $\boldit{k}_i$ has zeros `before' the appearance of $P_i$). The configuration $Z$ is uniquely determined by the pair $(K_Z,\phi_Z)$ where
\[
K_Z=\Set{P_1,\ldots,P_h} \subset X
\]
and $\phi_Z$ is the function
\[
\begin{tikzcd}[row sep=small]
K_Z \arrow{rr}{\phi_Z} & & \displaystyle\coprod_{\boldit{k} \in \BZ_{\geq 0}^{\lvert \bn \rvert} \setminus \boldit{0}}(\BA^d)_0^{[\boldit{k}]} \\
P_i \arrow[mapsto]{rr} & & {[Z_i]}.
\end{tikzcd}
\]    
By construction, one has $\bn = \sum_{1\leq i\leq h} \boldit{k}_i = \sum_{1\leq i\leq h} \tau(\phi(P_i))$, which confirms the sought after equivalence between pairs $(K,\phi)$ and points $[Z] \in X^{[\bn]}$.
\end{proof}

\subsection{The case of curves}
We provide an explicit formula for the generating series $\mathsf Z^\lambda_{X}(\boldit{q})$ in case $X$ is a smooth curve, thus proving \Cref{thm:main-C} from the introduction. 

\begin{notation}\label{notation:hooks}
Given a box $ (i,j)\in \lambda$ in the Young diagram $\lambda$, we  define the \emph{hook} $H(i,j)$ as the collection of boxes in $\lambda$ below or to the right to the given box $(i,j)\in \lambda$, namely
\[
H(i,j)=\Set{(l,k)\in \lambda| l=i,\, k\geq j \mbox{ or } l\geq i,\, k=j},
\]
and we set $h(i,j)=\lvert H(i,j)\rvert$. This is the \emph{hooklength} at the box $(i,j)$.
\end{notation}

For every $\Box \in \lambda$, define a new variable
\[
p_\Box = \prod_{\Box' \in H(\Box)} q_{\Box'}.
\]
\begin{corollary}\label{cor:motive-curve-dn}
Let $C$ be a smooth curve, $\lambda$ a Young diagram.
There is an identity
\[
\mathsf Z^\lambda_{C}(\boldit{q}) = \prod_{\Box \in \lambda} \, \left(1-p_{\Box}\right)^{-[C]} = \prod_{\Box \in \lambda} \zeta_C(p_{\Box}).
\]
In particular, $\mathsf Z^\lambda_{C}(\boldit{q})$ is a rational function in the variables $q_{\Box}$.
\end{corollary}

\begin{proof}
We have that 
\[
(\BA^1)_0^{[\bn]} \cong \Spec \bfk
\]
is a reduced point. Therefore
\[
\mathsf Z^\lambda_{\mathsf{punctual},1}(\boldit{q}) = \sum_{\bn \in \BRPP^\lambda} \boldit{q}^{\bn} =  \prod_{\Box \in \lambda} \, \left(1-p_\Box\right)^{-1}
\]
where the second equality is the combinatorial identity 
proved by Gansner\footnote{See Stanley \cite[Proposition 18.3]{Stanley_theory_application_I_II} and  Hillman-Grassl   \cite[Theorem 1]{HG_reversed_plane_partitions} for the original computation of the generating series  counting reverse plane partitions when setting $q_\Box=q$ for all boxes $\Box\in \lambda$.} \cite[Thm.~5.1]{Gans_reverse}.
The result now follows from \Cref{thm:power-structure-doublenested}.
\end{proof}

Note that \Cref{cor:motive-curve-dn} generalises the rank 1 version of the main result of \cite{MR_nested_Quot}, dealing with the classical nested Hilbert scheme of points on a curve. For higher dimensional (still smooth) varieties, it seems hard to obtain similar closed formulas, e.g.~because even the nested Quot schemes are almost always singular \cite{MR_lissite, monavari2024hyperquot}.

Applying the homomorphism of power structures $\chi \colon K_0(\Var_\BC) \to \BZ$ given by the topological Euler characteristic, and setting $q = q_\Box$ for every $\Box \in \lambda$, we obtain 
\[
\sum_{\bn \in \BRPP^\lambda} \chi(C^{[\bn]})q^{\lvert \bn\rvert} = \prod_{\Box \in \lambda} \,\left(1-q^{h(\Box)}\right)^{-\chi(C)},
\]
which was proven in \cite[Thm.~2.11]{Mon_double_nested}.

\appendix

\section{Connectedness and dimension of \texorpdfstring{$C^{[\bn]}$}{C\^n}: direct proofs}\label{sec:connectedness}

We give here a direct proof of \Cref{cor:conn}.

\begin{prop}\label{prop:connected}
Let $C$ be a smooth curve. Let $\lambda$ be a partition, and fix $\bn \in \BRPP^\lambda$. The double nested Hilbert scheme $C^{[\bn]}$ is connected.
\end{prop}

\begin{proof}
The proof has two steps. 

\underline{ Step 1.} Fix a closed point $P\in C$ and a configuration
$Z=(Z_\Box)_{\Box\in\lambda}$ of $C$ with shape $\lambda$ and underlying reverse plane partition $\bn \in \BRPP^\lambda$. Let $m_{\Box}^P$ be the multiplicity of $P$ in $Z_{\Box}$. Each box $\Box\in\lambda$ defines a family of subschemes $\mathcal Z_\Box\subset C\times C$ of length $\bn_{\Box}$ given by
\[
 \mathcal Z_\Box=
\begin{cases}
        \Set{(Q,Z_\Box - m_{\Box}^P\cdot P + m_{\Box}^P \cdot Q)|Q\in C} &\mbox{if } P\in \Supp Z_\Box\\
        C\times Z_\Box &\mbox{otherwise},
\end{cases}
\]
flat over the first factor of $C \times C$. These families are nested according to $\lambda$. Therefore, they define a flat family $\mathcal{Z} $ of configurations of $C$ with fibre $[Z]$ over the point $P\in C$. Notice that, all the members of the family $\mathcal Z$ have the same reverse plane partition as $Z$. This induces, via the universal property of $C^{[\bn]}$, a moduli morphism 
\[
\begin{tikzcd}
    C\arrow[r,"\varphi_{Z,P}"] &C^{[\bn]}
\end{tikzcd} 
\]
sending $Q \in C$ to the nested tuple $((\CZ_{\Box} \to C)^{-1}(Q))_{\Box \in \lambda}$ over $Q$. Its image is a quasiprojective, possibly singular, connected curve as soon as $P \in \Supp \bigcup_{\Box \in \lambda} Z_\Box$.

\smallbreak

\underline{Step 2.}
Let $P \in C$ be a closed point. Set $Z_P = (\boldit{n}_{\Box}\cdot P)_{\Box \in \lambda}$. Fix a point $[Z_1] \in C^{[\bn]}$. Our goal is to connect $[Z_P]$ and $[Z_1]$ via a chain of curves $C_1,\ldots,C_t \subset C^{[\bn]}$, for some $t$. Set
\[
S_1 = \Supp \left(\bigcup_{\Box} Z_{1,\Box}\right) = \Set{P_1,\ldots,P_s}\subset C.
\]
The identities
\[
Z_{1,\Box} = \sum_{1\leq i\leq s} m_{\Box}^{(i)}\cdot P_i
\]
define the multiplicities of $P_1,\ldots,P_s$ in each $Z_{1,\Box} \subset C$.
If $S_1 = \Set{P}$, we have nothing to prove. Up to relabeling the points in $S_1$, we may assume $P \neq P_1$. Construct the curve
\[
C_1 = \varphi_{Z_1,P_1}(C) \subset C^{[\bn]}.
\]
By definition, we have
\[
[Z_1] = \varphi_{Z_1,P_1}(P_1) \in C_1.
\]
Moreover, we define
\[
[Z_2] = \varphi_{Z_1,P_1}(P) \in C_1,
\]
which explicitly reads
\[
Z_{2,\Box} = Z_{1,\Box} - m_{\Box}^{(1)}\cdot P_1 + m_{\Box}^{(1)}\cdot  P,
\]
and, therefore,
\[
S_2 = \Supp \left(\bigcup_{\Box \in \lambda} Z_{2,\Box} \right) = \Set{P_2,\ldots,P_s} \cup \set{P}.
\]
Iterating this process $t \leq s$ times in total, we define points $[Z_3], \ldots, [Z_t] \in C^{[\bn]}$ via
\[
[Z_i] = \varphi_{Z_{i-1},P_{i-1}}(P),
\]
and at the same time curves $C_2,\ldots,C_{t-1} \subset C^{[\bn]}
$ via
\[
C_i=\varphi_{Z_i, P_i}(C),
\]
satisfying $[Z_i] \in C_i \cap C_{i-1}$. Finally,
\[
S_t = \Supp \left( \bigcup_{\Box \in \lambda}  Z_{t,\Box} \right) = \Set{P}.
\]
Thus $[Z_t] = [Z_P]$ and we are done.
\end{proof}

We now give a direct proof of the dimension formula given in \Cref{thm:irr components}\,\ref{irr-cpt-2}.

\begin{prop}\label{prop: dimension}
Let $C$ be a smooth curve. Let $\lambda$ be a partition, and fix $\bn \in \BRPP^\lambda$. We have
\[
\dim C^{[\bn]}=\omega(\bn).
\]
\end{prop}
\begin{proof}
We prove the statement by induction on the cardinality of the socle $\Soc(\lambda)$.
    
Suppose that the socle $\Soc(\lambda)$ contains only one box $\Box_0$. Then, we have a finite morphism 
\[ 
C^{[\bn]}\rightarrow C^{(\bn_{\Box_0})}
\]
to the symmetric product, which forgets all the boxes but the socle.
The base of the induction is now a consequence of \Cref{prop:weight-formula}.
    
Suppose now that the socle contains $k+1$ elements $\Box_1\preceq\cdots\preceq\Box_{k+1}$ totally ordered as elements of the socle.
First, notice that we can restrict to the case $\Box_1=(0,j_1)$. Indeed, if $\Box_1=(i_1,j_1)$, for some $i_1>0$, and we denote  by $\bn'\in\BRPP^{\lambda'}$ the reverse plane partition obtained removing the first $i_1-1$ columns from $\bn$ then, we have a finite morphism
    \[ 
    C^{[\bn]}\rightarrow C^{[\bn']}. 
    \]
Moreover, $\lvert\Delta \bn\rvert=\lvert\Delta \bn'\rvert$ again by \Cref{prop:weight-formula} (cf.~\Cref{def:derivative} for the definition of derivative), because they differ by subdiagrams with socle of cardinality one and the same label in the respective socle boxes.
So, we can suppose 
\[ 
\Box_1=(0,j_1),\quad \Box_2=(i_2,j_2),
\]
for some $j_1>j_2$ and $i_2>0$.
Define a new Young diagram $\lambda'$ (and a reverse plane partition $\bn'\in\BRPP^{\lambda'}$), obtained from $\lambda$ by throwing away all the boxes with lattice coordinates $(0,j_2+1), \dots, (0,j_1)$. Then, we have
\begin{align*}
    C^{[\bn]}\cong C^{[\bn']}\times C^{(\bn_{0,j_2+1}-\bn_{0,j_2})} \times \dots \times C^{(\bn_{0,j_1}-\bn_{0,j_1-1})},
    \end{align*}
which implies that
\begin{align}\label{eqn: Case II}
    \dim C^{[\bn]}=\dim C^{[\bn']}+\bn_{0,j_1}-\bn_{0,j_2}=\omega(\bn).
\end{align}
The last equality is a direct consequence of the definition of derivative of a reverse plane partition and of the inductive hypothesis. 
\end{proof}

\begin{example}\label{example:dimension}
We describe with an explicit example the procedure employed in the proof of  \Cref{prop: dimension}. Consider the following reverse plane partition.
\[
\begin{tikzpicture}
    \node at (0,0) {\young(011,12,23,5)};
    \node at (-1,-0.03) {$\bn=$};
\end{tikzpicture}
\] Its derivative is 
   \[
\begin{tikzpicture}
    \node at (0,0) {\young(010,10,10,3)};
    \node at (-1.15,0.02) {$\Delta \bn=$};
\end{tikzpicture}
   \]
   and we have $\omega(\bn)=\lvert\Delta\bn\rvert=6$. Define, for $i=1,2,3$, the reverse plane partitions $\bn_i\in\BRPP^{\lambda_i}$ by
\begin{figure}[H]
    \centering
    $\begin{matrix}
\begin{tikzpicture}
    \node at (0,0) {\young(011,12,23)};
    \node at (-1.15,-0.03) {$\bn_1=$};
\end{tikzpicture}
    \end{matrix}\quad \begin{matrix}
\begin{tikzpicture}
    \node at (0,0) {\young(11,2,3)};
    \node at (-0.95,-0.03) {$\bn_2=$};
\end{tikzpicture}
    \end{matrix}\quad \begin{matrix}
\begin{tikzpicture}
    \node at (0,0) {\young(11)};
    \node at (-0.95,-0.03) {$\bn_3=$};
\end{tikzpicture}
    \end{matrix}$
\end{figure}
\noindent 
so that, following the proof of \Cref{prop: dimension}, we have
\begin{itemize}
    \item [$\circ$] an isomorphism $C^{[\bn]}\cong C^{[\bn_1]}\times C^{(3)}$,
    \item [$\circ$] a finite morphism $C^{[\bn_1]}\to C^{[\bn_{2}]}$
    \item [$\circ$] an isomorphism $C^{[\bn_2]}\cong C^{[\bn_3]}\times C^{(1)}\times C^{(1)}$, and finally
    \item [$\circ$] $\dim C^{[\bn_3]}=1$.
\end{itemize}
This yields
\[
\dim  C^{[\bn]}=1+(3-1)+ (5-2)=6=\omega(\bn).
\]
\end{example}

\section{A worked example}\label{app:computations}

Consider the reverse plane partition
\[ 
\begin{tikzpicture}
    \node at (0,0) { \young(003,025,355)};
    \node at (-0.95,-0.0) {$\bn=$};
\end{tikzpicture}
\]
with shape $\lambda=(3,3,3)$, size $\lvert\bn\rvert = 23$ and weight $\omega(\bn) = 5$ (cf.~\Cref{ex:big}).\blfootnote{The code for the explicit computation of the example is available in the file \href{www.paololella.it/software/workedExample-DNHS.m2}{\tt workedExample-DNHS.m2}. All methods for exploring the features discussed in Sections \ref{subsec:combinatorics}-\ref{sec:virtual} are available in the \textit{Macaulay2} package \href{www.paololella.it/software/DoubleNestedHilbertSchemes.m2}{\tt DoubleNestedHilbertSchemes.m2}}

\paragraph{\it Irreducible components}
There are 19 Young indicators with shape $\lambda$ (\Cref{fig:young indicators example}) leading to 15  factorisations. According to \Cref{thm: main thm for criterion of smooth}, eight of them correspond to smooth irreducible components, while the other seven correspond to singular ones as a non trivial relation among the indicators exists:
\begin{description}
\item[$T_1$] $\bn = 2\bnu_{14} + \bnu_{15} + \bnu_{17}   + \bnu_{18} $ smooth;
\item[$T_2$] $\bn = \bnu_{8}  + \bnu_{14} + 2\bnu_{15}  + \bnu_{17}  $ smooth;
\item[$T_3$] $\bn = \bnu_{12}  + 2\bnu_{14}  + 2\bnu_{18} $ smooth (\emph{standard factorisation})\footnote{Notice that there is no complete factorisation as the derivative $\Delta\bn$ has a negative entry: $\Delta\bn(2,2) = \bn(2,2)-\bn(1,2)-\bn(2,1)+\bn(1,1) = -3$.};
\item[$T_4$] $\bn = \bnu_{8}  + \bnu_{12}  + \bnu_{14}  + \bnu_{15}  + \bnu_{18}$ singular ($\varphi_V$ is not a bijection on points as $\bnu_{8} + \bnu_{15}  = \bnu_{14} +\bnu_{18}$, cf.~\Cref{prop: bijec comb});
\item[$T_5$] $\bn = 2\bnu_{8}  + \bnu_{12}  +  2\bnu_{15} $ smooth;
\item[$T_6$] $\bn = \bnu_{13}  + \bnu_{14} + \bnu_{15}  + 2\bnu_{17}  $ smooth;
\item[$T_7$] $\bn = \bnu_{12}  + \bnu_{13}  +  \bnu_{14}  + \bnu_{17}  + \bnu_{18} $  singular ($\varphi_V$ is not a bijection on points as $\bnu_{13}  + \bnu_{17} = \bnu_{14} + \bnu_{18}$, cf.~\Cref{prop: bijec comb});
\item[$T_8$] $\bn = \bnu_{8}  + \bnu_{12}  + \bnu_{13}   + \bnu_{15} +  \bnu_{17}  $  singular ($\varphi_V$ is not a bijection on points since $\bnu_{8}    + \bnu_{15} = \bnu_{13} + \bnu_{17}$, cf.~\Cref{prop: bijec comb});
\item[$T_9$] $\bn = \bnu_{8}  +  2\bnu_{12}  + \bnu_{13}  +  \bnu_{18} $ smooth;
\item[$T_{10}$] $\bn = \bnu_{12}  + 2\bnu_{13}  +  2\bnu_{17} $ smooth;
\item[$T_{11}$] $\bn = \bnu_{7}  + \bnu_{12}  + \bnu_{14} + \bnu_{15}  + \bnu_{17}  $ singular ($\varphi_V$ is a bijection on points but the differential $\mathrm{d}\varphi_V$ is not injective as $\bnu_{7}  + 2\bnu_{12}  = \bnu_{14} + \bnu_{15} + \bnu_{17} $, cf.~\Cref{prop: inj diff comb} and \Cref{ex:big});
\item[$T_{12}$] $\bn = \bnu_{7}  + 2\bnu_{12}  + \bnu_{14}  + \bnu_{18} $  singular ($\varphi_V$ is not a bijection of points as $\bnu_{7}  + \bnu_{12}  = \bnu_{14}  + \bnu_{18} $, cf.~\Cref{prop: bijec comb});
\item[$T_{13}$] $\bn = \bnu_{7} + \bnu_{8}  + 2\bnu_{12}   + \bnu_{15} $ singular ($\varphi_V$ is not a bijection of points as $\bnu_{7}  + \bnu_{12}  =\bnu_{8}  + \bnu_{15} $, cf.~\Cref{prop: bijec comb});
\item[$T_{14}$] $\bn = \bnu_{7}   + 2\bnu_{12} + \bnu_{13}  + \bnu_{17} $  singular ($\varphi_V$ is not a bijection of points as $\bnu_{7}  + \bnu_{12} =  \bnu_{13}  + \bnu_{17}$, cf.~\Cref{prop: bijec comb});
\item[$T_{15}$] $\bn = 2\bnu_{7}  + 3\bnu_{12} $ smooth.
\end{description}

\begin{figure}
\begin{center}
\begin{tikzpicture}
\begin{scope}[shift={(0,0)}]
    \node at (-0.95,0) {$\bnu_1=$};
    \node at (0,0) {\tiny\young(111,111,111)};
\end{scope}

\begin{scope}[shift={(2.1,0)}]
    \node at (-0.95,0) {$\bnu_2=$};
    \node at (0,0) {\tiny\young(000,111,111)};
\end{scope}

\begin{scope}[shift={(4.2,0)}]
    \node at (-0.95,0) {$\bnu_3=$};
    \node at (0,0) {\tiny\young(000,000,111)};
\end{scope}

\begin{scope}[shift={(6.3,0)}]
    \node at (-0.95,0) {$\bnu_4=$};
    \node at (0,0) {\tiny\young(011,011,011)};
\end{scope}

\begin{scope}[shift={(8.4,0)}]
    \node at (-0.95,0) {$\bnu_5=$};
    \node at (0,0) {\tiny\young(011,111,111)};
\end{scope}

\begin{scope}[shift={(10.5,0)}]
    \node at (-0.95,0) {$\bnu_6=$};
    \node at (0,0) {\tiny\young(011,011,111)};
\end{scope}

\begin{scope}[shift={(12.6,0)}]
    \node at (-0.95,0) {$\bnu_7=$};
    \node at (0,0) {\tiny\young(000,011,011)};
\end{scope}

\begin{scope}[shift={(0 + 0.5,-1.2)}]
    \node at (-0.95,0) {$\bnu_8=$};
    \node at (0,0) {\tiny\young(000,011,111)};
\end{scope}

\begin{scope}[shift={(2.3+ 0.5,-1.2)}]
    \node at (-0.95,0) {$\bnu_9=$};
    \node at (0,0) {\tiny\young(000,000,011)};
\end{scope}

\begin{scope}[shift={(4.6+ 0.5,-1.2)}]
    \node at (-1,0) {$\bnu_{10}=$};
    \node at (0,0) {\tiny\young(001,001,001)};
\end{scope}

\begin{scope}[shift={(6.9+ 0.5,-1.2)}]
    \node at (-1,0) {$\bnu_{11}=$};
    \node at (0,0) {\tiny\young(001,111,111)};
\end{scope}

\begin{scope}[shift={(9.2+ 0.5,-1.2)}]
    \node at (-1,0) {$\bnu_{12}=$};
    \node at (0,0) {\tiny\young(001,001,111)};
\end{scope}

\begin{scope}[shift={(11.5+ 0.5,-1.2)}]
    \node at (-1,0) {$\bnu_{13}=$};
    \node at (0,0) {\tiny\young(001,011,011)};
\end{scope}

\begin{scope}[shift={(0 + 0.5,-2.4)}]
    \node at (-1,0) {$\bnu_{14}=$};
    \node at (0,0) {\tiny\young(001,011,111)};
\end{scope}

\begin{scope}[shift={(2.3+ 0.5,-2.4)}]
    \node at (-1,0) {$\bnu_{15}=$};
    \node at (0,0) {\tiny\young(001,001,011)};
\end{scope}

\begin{scope}[shift={(4.6+ 0.5,-2.4)}]
    \node at (-1,0) {$\bnu_{16}=$};
    \node at (0,0) {\tiny\young(000,001,001)};
\end{scope}

\begin{scope}[shift={(6.9+ 0.5,-2.4)}]
    \node at (-1,0) {$\bnu_{17}=$};
    \node at (0,0) {\tiny\young(000,001,111)};
\end{scope}

\begin{scope}[shift={(9.2+ 0.5,-2.4)}]
    \node at (-1,0) {$\bnu_{18}=$};
    \node at (0,0) {\tiny\young(000,001,011)};
\end{scope}

\begin{scope}[shift={(11.5+ 0.5,-2.4)}]
    \node at (-1,0) {$\bnu_{19}=$};
    \node at (0,0) {\tiny\young(000,000,001)};
\end{scope}
\end{tikzpicture}
\end{center}
\caption{\label{fig:young indicators example} The set of Young indicators with shape $\lambda = (3,3,3)$.}
\end{figure}

\paragraph{\it Local equations} Assume $C = \mathbb{A}^1$. Using {\em Type I} equations, the Hilbert scheme $(\mathbb{A}^1)^{[\bn]}$ is embedded in the product of classical Hilbert schemes
\[
\left((\mathbb{A}^1)^{(2)}\right) \times \left((\mathbb{A}^1)^{(3)}\right)^2 \times \left((\mathbb{A}^1)^{(5)}\right)^3 \simeq \mathbb{A}^{\lvert\bn\rvert} = \mathbb{A}^{23}. 
\]
The universal family of $(\mathbb{A}^1)^{[\bn]}$ is described by
\[
\begin{tikzpicture}[xscale=4.75,yscale=-1]
\node at (0,0) [] {\footnotesize $\mathcal{Z}_{0,0} = \varnothing = V(1)$};
\node at (1,0) [] {\footnotesize $\mathcal{Z}_{1,0} = \varnothing = V(1)$};
\node at (2,0) [] {\footnotesize $\mathcal{Z}_{2,0} = V\left(x^3 + \sum_{k=1}^3 a_{2,0,k}x^{3-k}\right)$};

\node at (0,0.5) [rotate=-90] {\small $\subset$};
\node at (0,1.5) [rotate=-90] {\small $\subset$};

\node at (0.5,0) [] {\small $\subset$};
\node at (1.5,0) [] {\small $\subset$};

\node at (0,1) [] {\footnotesize $\CZ_{0,1} = \varnothing = V(1)$};
\node at (1,1) [] {\footnotesize $\mathcal{Z}_{1,1} = V\left(x^2 + \sum_{k=1}^2 a_{1,1,k}x^{2-k}\right)$};
\node at (2,1) [] {\footnotesize $\mathcal{Z}_{2,1} = V\left(x^5 + \sum_{k=1}^5 a_{2,1,k}x^{5-k}\right)$};

\node at (0.5,1) [] {\small $\subset$};
\node at (1.5,1) [] {\small $\subset$};

\node at (1,0.5) [rotate=-90] {\small $\subset$};
\node at (1,1.5) [rotate=-90] {\small $\subset$};

\node at (0,2) [] {\footnotesize $\mathcal{Z}_{0,2} =V\left(x^3 + \sum_{k=1}^3 a_{0,2,k}x^{3-k}\right)$};
\node at (1,2) [] {\footnotesize $\mathcal{Z}_{1,2} =V\left(x^5 + \sum_{k=1}^5 a_{1,2,k}x^{5-k}\right)$};
\node at (2,2) [] {\footnotesize $\mathcal{Z}_{2,2} = V\left(x^5 + \sum_{k=1}^5 a_{2,2,k}x^{5-k}\right)$};

\node at (0.5,2) [] {\small $\subset$};
\node at (1.5,2) [] {\small $\subset$};

\node at (2,0.5) [rotate=-90] {\small $\subset$};
\node at (2,1.5) [rotate=-90] {\small $\subset$};

\end{tikzpicture}
\]
where the coefficients $a_{i,j,k}$ satisfy the conditions imposed via \eqref{eq: type i}
{\footnotesize
\begin{equation}\label{eq:type i example}
\begin{split}
 & \mathcal{Z}_{1,1} \subset \mathcal{Z}_{2,1}\ \left\{\hspace{-4pt}\begin{array}{l}
 a_{2,1,4}-a_{1,1,2}a_{2,1,2}+a_{1,1,2}^{2}-a_{1,1,1}a_{2,1,3}+2\,a_{1,1,1}a_{1,1,2}a_{2,1,1}+a_{1,1,1}^{2}a_{2,1,2}-3\,a_{1,1,1}^{2}a_{1,1,2}-a_{1,1,1}^{3}a_{2,1,1}+a_{1,1,1}^{4} = 0 \\[-3pt]
a_{2,1,5}-a_{1,1,2}a_{2,1,3}+a_{1,1,2}^{2}a_{2,1,1}+a_{1,1,1}a_{1,1,2}a_{2,1,2}-2\,a_{1,1,1}a_{1,1,2}^{2}-a_{1,1,1}^{2}a_{1,1,2}a_{2,1,1}+a_{1,1,1}^{3}a_{1,1,2} = 0\end{array}\right.
 \\
 & \mathcal{Z}_{2,0} \subset \mathcal{Z}_{2,1}\ \left\{\hspace{-4pt}\begin{array}{l}
 a_{2,1,3}-a_{2,0,3}-a_{2,0,2}a_{2,1,1}-a_{2,0,1}a_{2,1,2}+2\,a_{2,0,1}a_{2,0,2}+a_{2,0,1}^{2}a_{2,1,1}-a_{2,0,1}^{3} = 0 \\[-3pt]
a_{2,1,4}-a_{2,0,3}a_{2,1,1}-a_{2,0,2}a_{2,1,2}+a_{2,0,2}^{2}+a_{2,0,1}a_{2,0,3}+a_{2,0,1}a_{2,0,2}a_{2,1,1}-a_{2,0,1}^{2}a_{2,0,2} = 0 \\[-3pt]
a_{2,1,5}-a_{2,0,3}a_{2,1,2}+a_{2,0,2}a_{2,0,3}+a_{2,0,1}a_{2,0,3}a_{2,1,1}-a_{2,0,1}^{2}a_{2,0,3} = 0
\end{array}\right.
 \\
 &  \mathcal{Z}_{0,2} \subset \mathcal{Z}_{1,2}\ \left\{\hspace{-4pt}\begin{array}{l}
 a_{1,2,3}-a_{0,2,3}-a_{0,2,2}a_{1,2,1}-a_{0,2,1}a_{1,2,2}+2\,a_{0,2,1}a_{0,2,2}+a_{0,2,1}^{2}a_{1,2,1}-a_{0,2,1}^{3} = 0 \\[-3pt]
a_{1,2,4}-a_{0,2,3}a_{1,2,1}-a_{0,2,2}a_{1,2,2}+a_{0,2,2}^{2}+a_{0,2,1}a_{0,2,3}+a_{0,2,1}a_{0,2,2}a_{1,2,1}-a_{0,2,1}^{2}a_{0,2,2} = 0 \\[-3pt]
a_{1,2,5}-a_{0,2,3}a_{1,2,2}+a_{0,2,2}a_{0,2,3}+a_{0,2,1}a_{0,2,3}a_{1,2,1}-a_{0,2,1}^{2}a_{0,2,3} = 0
\end{array}\right.
 \\
 & \mathcal{Z}_{1,1} \subset \mathcal{Z}_{1,2}\ \left\{\hspace{-4pt}\begin{array}{l}
 a_{1,2,4}-a_{1,1,2}a_{1,2,2}+a_{1,1,2}^{2}-a_{1,1,1}a_{1,2,3}+2\,a_{1,1,1}a_{1,1,2}a_{1,2,1}+a_{1,1,1}^{2}a_{1,2,2}-3\,a_{1,1,1}^{2}a_{1,1,2}-a_{1,1,1}^{3}a_{1,2,1}+a_{1,1,1}^{4} = 0 \\[-3pt]
a_{1,2,5}-a_{1,1,2}a_{1,2,3}+a_{1,1,2}^{2}a_{1,2,1}+a_{1,1,1}a_{1,1,2}a_{1,2,2}-2\,a_{1,1,1}a_{1,1,2}^{2}-a_{1,1,1}^{2}a_{1,1,2}a_{1,2,1}+a_{1,1,1}^{3}a_{1,1,2} = 0 \\[-3pt]
\end{array}\right.
 \\[2pt]
  &  \mathcal{Z}_{1,2} \subset \mathcal{Z}_{2,2}\ \left\{ a_{2,2,1}-a_{1,2,1} = a_{2,2,2}-a_{1,2,2} = a_{2,2,3}-a_{1,2,3} = a_{2,2,4}-a_{1,2,4} = a_{2,2,5}-a_{1,2,5} = 0\right.
  \\
 &  \mathcal{Z}_{2,1} \subset \mathcal{Z}_{2,2}\ \left\{ a_{2,2,1}-a_{2,1,1} =
a_{2,2,2}-a_{2,1,2} =
a_{2,2,3}-a_{2,1,3} = a_{2,2,4}-a_{2,1,4} =
a_{2,2,5}-a_{2,1,5} = 0\right.
\end{split}
\end{equation}
}

Using {\em Type II} equations, the Hilbert scheme $(\mathbb{A}^1)^{[\bn]}$ is embedded in the product of classical Hilbert schemes
\[
\left((\mathbb{A}^1)^{(2)} \times (\mathbb{A}^1)^{(2)} \right) \times \left((\mathbb{A}^1)^{(2)} \times (\mathbb{A}^1)^{(3)} \right)^2 \times \left((\mathbb{A}^1)^{(3)} \times (\mathbb{A}^1)^{(3)} \right)^2 \simeq \mathbb{A}^{26} 
\]
and the universal family is described by
\[
\begin{tikzpicture}[xscale=5.3,yscale=-1.5]
\node at (0,0) [] {\scriptsize $\mathcal{Z}_{0,0} = \varnothing = V(1)$};
\node at (0.95,0) [] {\scriptsize $\mathcal{Z}_{1,0} = \varnothing = V(1)$};
\node at (1.9,0) [] {\scriptsize $\mathcal{Z}_{2,0} = \left\{\begin{array}{l} \hspace{-4pt}\mathcal{Z}_{1,0} + V\left(x^3 + \sum_{k=1}^3 b_{2,0,k} x^{3-k} \right)\\ \hspace{-4pt} V\left( x^3 + \sum_{k=1}^3 c_{2,0,k}x^{3-k}\right)\end{array}\right.$};

\node at (0,0.5) [rotate=-90] {\small $\subset$};
\node at (0,1.5) [rotate=-90] {\small $\subset$};

\node at (0.45,0) [] {\small $\subset$};
\node at (1.4,0) [] {\small $\subset$};

\node at (0,1) [] {\scriptsize $\mathcal{Z}_{0,1} = \varnothing = V(1)$};
\node at (0.95,1) [] {\scriptsize $\mathcal{Z}_{1,1} = \left\{\begin{array}{l} \hspace{-4pt}\mathcal{Z}_{0,1} + V\left(x^2 +\sum_{k=1}^2 b_{1,1,k} x^{2-k} \right)\\ \hspace{-4pt} \mathcal{Z}_{1,0} + V\left( x^2 +\sum_{k=1}^2 c_{1,1,k} x^{2-k}\right)\end{array}\right.$};
\node at (1.9,1) [] {\scriptsize $\mathcal{Z}_{2,1} = \left\{\begin{array}{l} \hspace{-4pt}\mathcal{Z}_{1,1} + V\left(x^3 + \sum_{k=1}^3 b_{2,1,k}x^{3-k} \right)\\ \hspace{-4pt}\mathcal{Z}_{2,0} + V\left( x^2 + \sum_{k=1}^2 c_{2,1,k}x^{2-k}\right)\end{array}\right.$};

\node at (0.45,1) [] {\small $\subset$};
\node at (1.4,1) [] {\small $\subset$};

\node at (0.95,0.5) [rotate=-90] {\small $\subset$};
\node at (0.95,1.5) [rotate=-90] {\small $\subset$};

\node at (0,2) [] {\scriptsize $\mathcal{Z}_{0,2} = \left\{\begin{array}{l} \hspace{-4pt}V\left( x^3 + \sum_{k=1}^3 b_{0,2,k}x^{3-k}\right) \\ \hspace{-4pt}\mathcal{Z}_{0,1} + V\left(x^3 + \sum_{k=1}^3 c_{0,2,k} x^{3-k} \right)\end{array}\right.$};
\node at (0.95,2) [] {\scriptsize $\mathcal{Z}_{1,2} = \left\{\begin{array}{l} \hspace{-4pt}\mathcal{Z}_{0,2} + V\left(x^2 + \sum_{k=1} b_{1,2,k}x^{2-k} \right)\\ \hspace{-4pt}\mathcal{Z}_{1,1} + V\left( x^3 + \sum_{k=1}^3 c_{1,2,k}x^{3-k}\right)\end{array}\right.$};
\node at (1.9,2) [] {\scriptsize $\mathcal{Z}_{2,2} = \left\{\begin{array}{l} \hspace{-4pt}\mathcal{Z}_{1,2} + V\left(1 \right)\\ \hspace{-4pt}\mathcal{Z}_{2,1} + V\left( 1\right)\end{array}\right.$};

\node at (0.45,2) [] {\small $\subset$};
\node at (1.4,2) [] {\small $\subset$};

\node at (1.9,0.5) [rotate=-90] {\small $\subset$};
\node at (1.9,1.5) [rotate=-90] {\small $\subset$};

\end{tikzpicture}
\]
where the coefficients $b_{i,j,k}$ and $c_{i,j,k}$ satisfy the conditions imposed via \eqref{eq: type ii}
{\footnotesize
\begin{equation}\label{eq:type ii example}
\begin{split}
    \mathcal{Z}_{2,0} = (\mathcal{Z}_{2,0} - \mathcal{Z}_{1,0}) + \mathcal{Z}_{1,0}&\ \left\{c_{2,0,1}-b_{2,0,1} = c_{2,0,2}-b_{2,0,2} = c_{2,0,3}-b_{2,0,3} = 0\right.
    \\[2pt]
    (\mathcal{Z}_{1,1} - \mathcal{Z}_{0,1}) + (\mathcal{Z}_{0,1}-\mathcal{Z}_{0,0}) = (\mathcal{Z}_{1,1} - \mathcal{Z}_{1,0}) + (\mathcal{Z}_{1,0}-\mathcal{Z}_{0,0})&\ \left\{ c_{1,1,1}-b_{1,1,1} =
c_{1,1,2}-b_{1,1,2}=0\right.\\
    (\mathcal{Z}_{2,1} - \mathcal{Z}_{1,1}) + (\mathcal{Z}_{1,1}-\mathcal{Z}_{1,0}) = (\mathcal{Z}_{2,1} - \mathcal{Z}_{2,0}) + (\mathcal{Z}_{2,0}-\mathcal{Z}_{1,0})&\ \left\{\hspace{-4pt}\begin{array}{l} 
    -c_{2,1,1}+c_{1,1,1}+b_{2,1,1}-b_{2,0,1}=0\\[-3pt]
-c_{2,1,2}+c_{1,1,2}+b_{2,1,2}-b_{2,0,2}+b_{2,1,1}c_{1,1,1}-b_{2,0,1}c_{2,1,1}=0\\[-3pt]
b_{2,1,3}-b_{2,0,3}+b_{2,1,2}c_{1,1,1}+b_{2,1,1}c_{1,1,2}-b_{2,0,2}c_{2,1,1}-b_{2,0,1}c_{2,1,2}=0\\[-3pt]
b_{2,1,3}c_{1,1,1}+b_{2,1,2}c_{1,1,2}-b_{2,0,3}c_{2,1,1}-b_{2,0,2}c_{2,1,2}=0\\[-3pt]
b_{2,1,3}c_{1,1,2}-b_{2,0,3}c_{2,1,2}=0
    \end{array}\right.\\
     \mathcal{Z}_{0,2} = (\mathcal{Z}_{0,2} - \mathcal{Z}_{0,1}) + \mathcal{Z}_{0,1}&\ \left\{ c_{0,2,1}-b_{0,2,1} = c_{0,2,2}-b_{0,2,2} = c_{0,2,3}-b_{0,2,3}=0\right.\\
   (\mathcal{Z}_{1,2} - \mathcal{Z}_{0,2}) + (\mathcal{Z}_{0,2}-\mathcal{Z}_{0,1}) = (\mathcal{Z}_{1,2} - \mathcal{Z}_{1,1}) + (\mathcal{Z}_{1,1}-\mathcal{Z}_{0,1}) &\ \left\{ \hspace{-4pt}\begin{array}{l} 
   -c_{1,2,1}+c_{0,2,1}+b_{1,2,1}-b_{1,1,1}=0\\[-3pt]
-c_{1,2,2}+c_{0,2,2}+b_{1,2,2}-b_{1,1,2}+b_{1,2,1}c_{0,2,1}-b_{1,1,1}c_{1,2,1}=0\\[-3pt]
-c_{1,2,3}+c_{0,2,3}+b_{1,2,2}c_{0,2,1}+b_{1,2,1}c_{0,2,2}-b_{1,1,2}c_{1,2,1}-b_{1,1,1}c_{1,2,2}=0\\[-3pt]
b_{1,2,2}c_{0,2,2}+b_{1,2,1}c_{0,2,3}-b_{1,1,2}c_{1,2,2}-b_{1,1,1}c_{1,2,3}=0\\[-3pt]
b_{1,2,2}c_{0,2,3}-b_{1,1,2}c_{1,2,3}=0
   \end{array}\right.\\
    (\mathcal{Z}_{2,2} - \mathcal{Z}_{1,2}) + (\mathcal{Z}_{1,2}-\mathcal{Z}_{1,1}) = (\mathcal{Z}_{2,2} - \mathcal{Z}_{2,1}) + (\mathcal{Z}_{2,1}-\mathcal{Z}_{1,1})& \ \left\{
    c_{1,2,1}-b_{2,1,1} = c_{1,2,2}-b_{2,1,2} = c_{1,2,3}-b_{2,1,3} = 0\right.
\end{split}
\end{equation}
}

We notice that the isomorphism between the two representing schemes can be explicitly described looking at any decomposition of a divisor $\mathcal{Z}_{i,j}$ as sums of consecutive differences leading to $\mathcal{Z}_{0,0}$. For instance, 
{\footnotesize
\begin{align*}
    & \mathcal{Z}_{2,2} = (\mathcal{Z}_{2,2} - \mathcal{Z}_{1,2}) + (\mathcal{Z}_{1,2} - \mathcal{Z}_{1,1}) + (\mathcal{Z}_{1,1} - \mathcal{Z}_{0,1}) + (\mathcal{Z}_{0,1} - \mathcal{Z}_{0,0}) + \mathcal{Z}_{0,0}\\[-2pt]
    &\hspace{1cm}\Leftrightarrow x^5 + \textstyle\sum_{k=1}^5 a_{2,2,k}x^{5-k} = \left(x^3+\sum_{k=1}^3c_{1,2,k}x^{3-k}\right)\left(x^2+\sum_{k=1}^2b_{1,1,k}x^{2-k}\right) \\[-4pt]
    &\hspace{2cm}\Leftrightarrow\quad \left\{\begin{array}{l} a_{2,2,1} = c_{1,2,1}+b_{1,1,1}\\[-4pt] 
    a_{2,2,2} = c_{1,2,2}+b_{1,1,2}+b_{1,1,1}c_{1,2,1}\\[-4pt]
    a_{2,2,3} = c_{1,2,3}+b_{1,1,2}c_{1,2,1}+b_{1,1,1}c_{1,2,2}\\[-4pt] 
    a_{2,2,4} =  b_{1,1,2}c_{1,2,2}+b_{1,1,1}c_{1,2,3}\\[-4pt] 
    a_{2,2,5} = b_{1,1,2}c_{1,2,3}\end{array}\right.
\end{align*}
}
\paragraph{\it Embedding in the tangent space at the origin} The tangent space to $(\mathbb{A}^1)^{[\bn]}$ at the origin is 9-dimensional and its equations are the linear parts of equations in \eqref{eq:type i example} or in \eqref{eq:type ii example}. In both cases, the Hilbert scheme $(\mathbb{A}^1)^{[\bn]}$ in $T_{[0]} (\mathbb{A}^1)^{[\bn]} \simeq \mathbb{A}^9$ is cut out by 4 equations: 2 of degree 4 and 2 of degree 5 (see \Cref{rk:equations-degree}).

\newpage

\ifx\undefined\bysame
\newcommand{\bysame}{\leavevmode\hbox to3em{\hrulefill}\,}
\fi

\vspace{\stretch{1}}

\noindent
{\small Michele Graffeo \\
\address{SISSA, Via Bonomea 265, 34136, Trieste (Italy)} \\
\href{mailto:mgraffeo@sissa.it}{\texttt{mgraffeo@sissa.it}}
}

\smallskip

\noindent
{\small Paolo Lella \\
\address{Dipartimento di Matematica, Politecnico di Milano,  Piazza Leonardo da Vinci 32, 20133 Milano (Italy)} \\
\href{mailto:paolo.lella@polimi.it}{\texttt{paolo.lella@polimi.it}}
}

\smallskip

\noindent
{\small Sergej Monavari \\
\address{\'Ecole Polytechnique F\'ed\'erale de Lausanne (EPFL),  CH-1015 Lausanne (Switzerland)} \\
\href{mailto:sergej.monavari@epfl.ch}{\texttt{sergej.monavari@epfl.ch}}
}

\smallskip

\noindent
{\small Andrea T. Ricolfi \\
\address{SISSA, Via Bonomea 265, 34136, Trieste (Italy)} \\
\href{mailto:aricolfi@sissa.it}{\texttt{aricolfi@sissa.it}}
}

\smallskip

\noindent
{\small Alessio Sammartano \\
\address{Dipartimento di Matematica, Politecnico di Milano,  Piazza Leonardo da Vinci 32, 20133 Milano (Italy)} \\
\href{mailto:alessio.sammartano@polimi.it}{\texttt{alessio.sammartano@polimi.it}}
}

\end{document}